\documentclass[10pt]{article}
\usepackage{amsmath,amsthm,amssymb,latexsym,amscd}
\usepackage{epsf}
\usepackage[pdftex]{graphicx} % commented this one
\usepackage{hyperref}%bacim to use and have hyper links and references in PDF; comment \usepackage[dvips]{graphicx} and
\newcommand{\e}{{\varepsilon}}

\newcommand{\xh}{ \hat{x}}

\newcommand{\yh}{\hat{y}}

\newcommand{\twoscale}[2]{#1 \stackrel{2}{\rightharpoonup} #2}
\newcommand{\xe}{\frac{x}{\e}}
\newcommand{\ki}[1]{\chi_{\mbox{\scriptsize #1}}}
\newcommand{\kie}[1]{\chi^{\e}_{\mbox{\scriptsize #1}}}
\renewcommand{\ss}[1]{\mbox{\scriptsize #1}}

\numberwithin{equation}{section}

\newtheorem{theorem}{Theorem}[section]
\newtheorem{lemma}[theorem]{Lemma}

\newtheorem{propo}[theorem]{Proposition}
\newtheorem{defn2}[theorem]{Definition}

\newtheorem{rem_no_numbers}[theorem]{Remark}

\title{Multiscale Analysis of Heterogeneous Media in the Peridynamic Formulation\thanks{The Authors acknowledge the support of:
Boeing Contract \# 207114,
AFOSR Grant FA 9550-05-1008, and
NSF Grant DMS-0406374}}

\author{Bacim Alali\thanks{Department of Mathematics,
University of Utah,
	Salt Lake City, UT,({\tt alali@math.utah.edu}).}
\and Robert Lipton\thanks{Department of 	Mathematics,
        Louisiana State University,
        Baton Rouge, LA 70803,
        ({\tt lipton@math.lsu.edu}).}}

\date{}

\begin{document}

\maketitle

\begin{abstract}
A methodology is presented for investigating the dynamics of heterogeneous  media
using the nonlocal continuum model given by the peridynamic formulation. The approach presented here provides the ability to model the macroscopic dynamics while at the same time resolving the dynamics at the length scales of the microstructure. Central to the methodology is  a novel two-scale evolution equation. The rescaled solution of this equation is shown to provide a strong approximation to the actual deformation inside the peridynamic material. The two scale evolution can  be split into a microscopic component tracking the dynamics at the length scale of the heterogeneities and  a macroscopic component tracking the volume averaged (homogenized) dynamics. The interplay between the microscopic and macroscopic dynamics is given by a coupled system of evolution equations. The  equations show that the forces generated by the homogenized deformation inside the medium are  related to the homogenized  deformation through a history dependent constitutive relation. 
\end{abstract}

%\tableofcontents

\section{Introduction}
\label{sec_into_perid}

The peridynamic formulation introduced in Silling \cite{Silling39} is a non-local continuum theory for deformable bodies.
Material particles interact through a pairwise force field that acts within a prescribed {\it horizon}. 
Interactions depend only on the difference in the displacement of material points and spatial derivatives in the displacement are avoided. 
This feature makes it an attractive model for the autonomous evolution of discontinuities in the displacement for
problems that involve cracks, interfaces, and other defects, see \cite{{Bobaru11},
{Bobaru12},{Gerstle22},{Silling40},{Silling42},{Silling43}}. 
Recent investigations aimed toward developing the numerical implementation, and application areas of the peridynamic model include \cite{Bhattacharya17}, \cite{Silling47}, \cite{Weckner50}, \cite{Weckner51}, \cite{Zimmerman53}. 
More mathematically related investigations  address issues related to the function space setting of peridynamics \cite{Emmrich21}, \cite{Du1} and the link between the linearized peridynamic formulation and the operators appearing in the Navier system of linear elasticity in the limit of vanishing non-locality \cite{Emmrich21}, \cite{Lehoucq1}. In this context the convergence of the solutions of the peridynamic equations to the solutions of the Navier system is demonstrated in \cite{Du1}. In other related work  the development of a non-local vector calculus with applications to non-local boundary value problems has been carried out in \cite{Lehoucq2}. Recent work on the  multi-scale applications of peridynamics have shown how the peridynamic equations formulated at mezo-scales can be recovered by a suitable upscaling of atomistic formulations, see \cite{Lehoucq3}. 

In this paper new tools are developed for the  analysis of heterogeneous peridynamic media involving two distinct length scales over which different types of peridynamic forces interact. 
The setting treated here involves a long range peridynamic force law perturbed in space by an oscillating short range peridynamic force. The oscillating short range force represents the presence of heterogeneities. It is also assumed that there is a sharp density variation associated with the heterogeneities. 
In this treatment we carry out the analysis in
the small deformation setting. For this case the reference and deformed configurations are taken to be the same
and both long and short range forces are given by linearizations of the peridynamic bond stretch model introduced in \cite{Silling39}.

The relative length scale over which the short range forces interact is denoted by $\e$ and points inside the domain
containing the heterogeneous material are specified by $x$. Here we will suppose the heterogeneities are periodically dispersed on the length scale $\e=\frac{1}{n}$ for some choice of $n=1,2,\ldots$ The deformation inside the medium is both a function of space and time $t$ and is written $u^\e(x,t)$.
The multi-scale analysis of the peridynamic formulation proceeds using the concept of two-scale convergence, introduced and developed by Nguetseng \cite{Nguetseng} and subsequently in  Allaire \cite{Allaire}, see also  E \cite{E}. The two-scale convergence  originally introduced in the context of partial differential equations  turns out to provide a natural  setting
for identifying both the coarse scale and fine scale dynamics inside peridynamic composites. The theory and application of the two-scale convergence is taken up in section three of this paper where a novel two-scale peridynamic equation is derived. The two-scale formulation is described by introducing a rescaled or microscopic variable $y=x/\e$. The solution of the two-scale dynamics is a deformation $u(x,y,t)$ that depends on both variables $x$ and $y$.

The rescaled solution $u(x,x/\e,t)$  is shown to provide a strong approximation to the actual deformation $u^\e(x,t)$ inside the peridynamic material. This is is shown in section \ref{ch_strongapproximation}
 where an evolution law for the error $e^\e(x,t)=u^\e(x,t)-u(x,x/\e,t)$ is developed. It is shown that $e^\e(x,t)$ vanishes in the $L^p$ norm, with respect to the spatial variables,  when the length scale of the oscillation tends to zero for all $p$ in the interval $1\leq p\leq\infty$.
The advantage of using the two-scale dynamics as a computational model is that it has the potential to lower computational costs associated with the explicit peridynamic modeling of millions of heterogeneities. This issue is discussed in section \ref{ch_strongapproximation}.

It is important for the modeling to recover the dynamics that can be measured by strain gages or other macroscopic measuring devices. Typical measured quantities involve averages of the deformation $u^\e(x,t)$ taken over a prescribed region $V$ with volume denoted by $|V|$. To this end we denote the unit period cell for the heterogeneities by $Y$ and project out the fluctuations by averaging over $y$ and write
\begin{eqnarray}
u^H(x,t)=\int_Y u(x,y,t)dy.
\label{avgerage}
\end{eqnarray}
In section four it is shown that 
\begin{eqnarray}
\lim_{\e\rightarrow 0}\frac{1}{|V|}\int_V u^\e(x,t)\,dx=\frac{1}{|V|}\int_V u^H(x,t)\,dx.
\label{volavg}
\end{eqnarray}
In this way we see that the average deformation is characterized by $u^H(x,t)$ when the scale $\e$  of the microstructure is small. We split the deformation into microscopic and macroscopic parts and write  $u(x,y,t)=u^H(x,t)+r(x,y,t)$. 
The interplay between the microscopic and macroscopic dynamics is given by a coupled system of evolution equations
for $u^H$ and $r$. The  equations  show that forces generated by the homogenized deformation inside the medium are  related to the homogenized  deformation through a history dependent constitutive relation. The explicit form of the constitutive relation is presented  in section four where we present a homogenized evolution equation for the coarse scale dynamics written exclusively in terms of $u^H$, see \eqref{id_u_H}.

%In this treatment the peridynamic operators and density fluctuations are allowed to be discontinuous on the length %scale of the heterogenieites. 
%Moreover the oscillations in the initial data and body force also can be discontinuous on the same length scale as %the heterogeneities. For this case we choose body forces and initial conditions that are continuous in the slow %variables and discontinuous in the fast variables. The analysis shows that in the limit of infinitely fine %oscillations the associated homogenized deformation $u^H$ is a continuous function while the oscillatory correction 
%$r$ is strongly discontinuous in the fast variable. These ideas are presented in sections two and three where the %function space setting is developed.

\begin{figure}[t]%[ht]
\centering
\scalebox{0.2}{\includegraphics{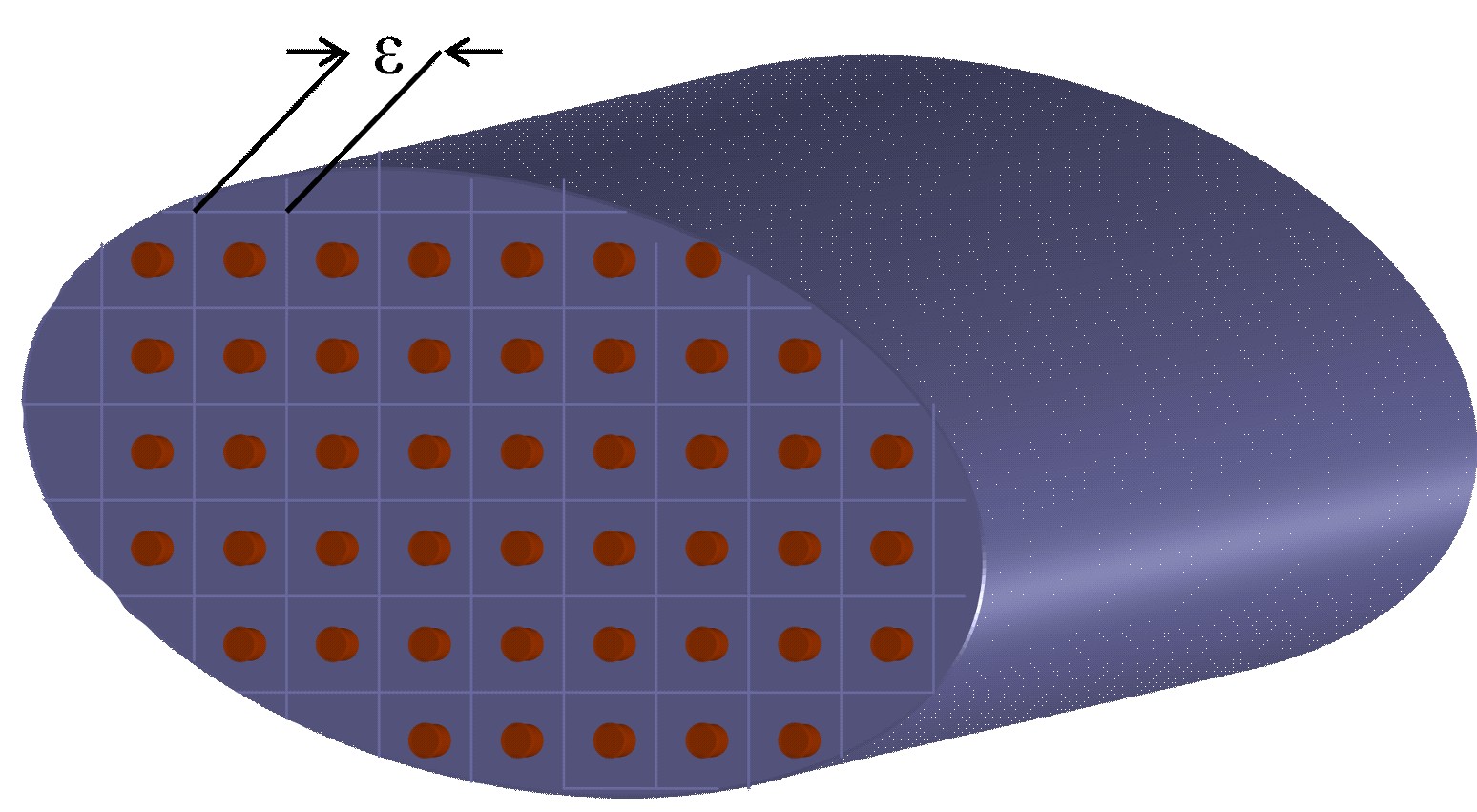}}
\caption{Fiber-reinforced composite.}
\label{Fi:fiber}
\end{figure}

\subsection{Peridynamic Formulation of Continuum Mechanics in Heterogeneous Media}
\label{sec_overview_perid}

We consider elastic deformations inside a body described by the bounded domain  $\Omega$.
In the peridynamic theory, the time evolution of the displacement vector field $u$, in a homogeneous body of constant density $\hat{\rho}$  is given by the partial
integro-differential equation
\begin{equation}
\label{peridynamicseq}
\hat{\rho}\;\partial_{t}^2 u(x,t) = \int_{H_\gamma(x)\cap\Omega} f(u(\xh,t)-u(x,t),\xh-x,x)\,d\xh
+ b(u,x,t),\,\,\,\,\,\, (x,t)\in\Omega\times(0,T)
\end{equation}
where $H_\gamma(x)$ is a neighborhood of $x$ of diameter $2\gamma$,
 $b$ is a prescribed loading force density field, and $\Omega$ is a bounded set in $\mathbb{R}^3$.
%Here $\delta$ denotes the peridynamic horizon and $H_\delta(x)\subset\Omega$ is a ball centered at $x$  with radius $\delta$.
Here $f$ denotes the pairwise force field whose value is the force vector (per unit volume squared) that the particle at $\hat{x}$
exerts on the particle at $x$.  For a homogeneous medium $f$ is of the form $f(u(\xh,t)-u(x,t),\xh-x)$, i.e., it depends only on the
relative position of the two particles. We will often refer to $f$ as a {\it bond force}.
Only points $\hat{x}$ inside $H_\gamma(x)$ interact with $x$.
Equation (\ref{peridynamicseq}) is supplemented with initial conditions
 for $u(x,0)$ and $\partial_t u(x,0)$. For the purposes
 of discussion it will be convenient
to set
\[
\xi=\xh-x,
\]
which represents the relative position of these two particles in the reference configuration, and
\[
\eta=u(\xh,t)-u(x,t),
\]
which represents their relative displacement (see Figure \ref{Fi:horizon}). 
\begin{figure}[t]%[ht]
\centering
\scalebox{0.4}{\includegraphics{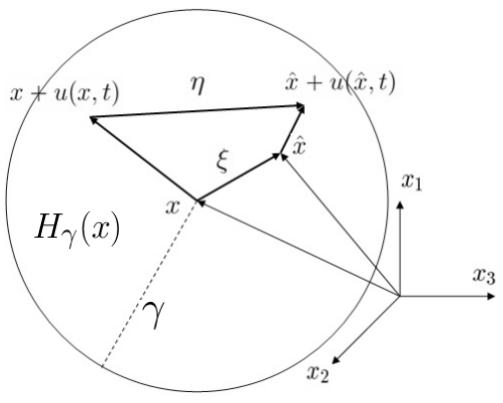}}
\caption{Deformation of a bond within the peridynamic horizon.}
\label{Fi:horizon}
\end{figure}
In this treatment, all elastic deformations are assumed small and the reference and deformed configurations are taken to be the same. 

We now introduce the heterogeneous peridynamic material. One can think of it as a  material  with long range peridynamic forces acting over a neighborhood of diameter $2\gamma$ perturbed by an oscillating  density fluctuation and oscillatory short range bond force acting over a much smaller neighborhood of diameter $2\e\delta$.
Both the long and short range pairwise elastic forces will be given by the linearized version of the {\it bond-stretch model} proposed in \cite{Silling43}. The long range force is given by
\begin{equation}
%\label{}
\nonumber
%f_{\mbox{{\tiny ply}}}(\eta,\xi) =\left\{
f_{\mbox{{\scriptsize  long}}}(\eta,\xi) =\left\{
\begin{array}{c l}
\displaystyle\lambda\frac{\xi\otimes\xi}{\vert\xi\vert^{3}} \eta, &  \vert\xi\vert\leq   \gamma    \\
0,   & \mbox{otherwise}.
\end{array}
\right.
\end{equation}
Here $\xi\otimes\xi$ is a rank one matrix with elements $(\xi\otimes\xi)_{ij}=\xi_i\xi_j$ and  $  \gamma  $ is the prescribed peridynamic horizon and $\lambda$ is a positive constant.

In this paper we assume that oscillations in the density and short range bond force are periodic. Here the oscillations are characterized by rescalings of a unit periodic peridynamic bond force and density. To describe these we introduce the unit period cube $Y\subset \mathbb{R}^3$ for the microstructure. The local coordinates inside $Y$ are denoted by $y$ with
the origin at the center of the unit cube. The unit cube is composed  of two or more peridynamic materials with different densities. To fix ideas one can consider reinforced composites made up of an inclusion phase such as a particle or fiber and a second host phase that surrounds the particle or fiber. A fiber reinforced material is portrayed in Figure \ref{Fi:Y}. The presence of material heterogeneity is reflected by the appearance of peridynamic forces acting within the length scale of the period. Let
$\chi_{\ss{f}}$ denote
the indicator function of the set occupied by the inclusion material and $\chi_{\ss{m}}$ denote the the indicator function
of the set occupied by the host or matrix material. Here $\chi_{\ss{f}}$ is given by
\begin{equation}
%\label{}
\nonumber
\chi_{{\scriptsize\mbox{f}}}(y) =\left\{
\begin{array}{c l}
1, & $y$ \mbox{ is in the inclusion phase},   \\
0,   & \mbox{otherwise,}
\end{array}
\right.
\end{equation}
and $\chi_{\ss{m}}$ is given by
\[
\chi_{\ss{m}}(y)=1-\chi_{\ss{f}}(y).
\]
We extend the functions $\chi_{\ss{f}}$ and $\chi_{\ss{m}}$ to $\mathbb{R}^3$ by periodicity.
For future reference, we denote by $\theta_{\ss{f}}$ and $\theta_{\ss{m}}$ the volume fractions of the included material and the matrix
material, respectively. Here $\theta_{\ss{f}}=\int_Y \chi_{\ss{f}}(y) dy$ and $\theta_{\ss{m}}=1-\theta_{\ss{f}}$. %Also, we let
%$n$ denote a unit vector parallel to the  fiber direction.
The density of the matrix material inside the unit period cell is given by the unperturbed density $\rho_{\ss{m}}=\hat{\rho}$ and that of the inclusion is given by $0<\rho_{\ss{f}}=\hat{\rho}+\Delta\rho$ where $\Delta\rho$ can be a positive or negative constant. The density characterizing the heterogeneous medium is  
\begin{eqnarray}
\rho(y)=\chi_{\ss{f}}(y)\rho_{\ss{f}}+\chi_{\ss{m}}(y)\rho_{\ss{m}}.
\label{denstyfluct}
\end{eqnarray}
The short-range pairwise force  is characterized by a bond strength $\alpha_\delta$ associated with a horizon $\delta>0$. The peridynamic horizon $\delta$ is chosen to be smaller than the spacing separating the inclusions. In addition the inclusions are assumed to be sufficiently smooth so that the points $y$ and $\yh$
are separated by at most one interface when $|y-\yh|<\delta$.
For any two points $y$ and $\hat{y}$ in $\mathbb{R}^3$  $\alpha_\delta$ is given by 
\begin{equation}
\nonumber
\label{alpha_short}
\displaystyle \alpha_\delta(y,\yh) =\left\{
\begin{array}{c l}
C_{\mbox{{\scriptsize f}}}, & \mbox{if $y$ and $\yh$ are in the same inclusion and $|y-\hat{y}|<\delta$}  \\
C_{\mbox{{\scriptsize m}}},  & \mbox{if $y$ and $\yh$ are in the matrix phase and $|y-\hat{y}|<\delta$}\\
C_{\mbox{{\scriptsize i}}},   & \mbox{if $y$ and $\yh$ are separated by an interface and $|y-\hat{y}|<\delta$}\\
0,&\mbox{if $|y-\hat{y}|\geq\delta$}.
\end{array}
\right.
\end{equation}
%\[
%\alpha(y,y+\xi_y) = C_{\mbox{{\tiny fiber}}} \chi_{{\tiny\mbox{fiber}}}(y)\chi_{{\tiny\mbox{fiber}}}(y+\xi_y) +
%C_{\mbox{{\tiny matrix}}} \chi_{{\tiny\mbox{matrix}}}(y)\chi_{{\tiny\mbox{matrix}}}(y+\xi_y)+
%C_{\mbox{{\tiny interface}}}\left(\chi_{{\tiny\mbox{fiber}}}(y)\chi_{{\tiny\mbox{matrix}}}(y+\xi_y)+
%\chi_{{\tiny\mbox{matrix}}}(y)\chi_{{\tiny\mbox{fiber}}}(y+\xi_y)\right).
%\]
The
material parameters $C_{\ss{f}}$ and $C_{\ss{m}}$ are intrinsic to each phase and can  be
determined through experiments. Bonds connecting particles in the different materials are characterized by
$C_{\ss{i}}$, which can be chosen such that
 $C_{\ss{f}}> C_{\ss{i}}> C_{\ss{m}}>0$, see \cite{Silling43}.
Mathematically we express the bond strength as 
\begin{eqnarray}
\alpha_\delta(y,\yh)=\chi_\delta(y-\yh)\alpha(y,\yh),
\label{microhoriz}
\end{eqnarray}
where $\chi_\delta(z)=1$ for $|z|<\delta$ and $\chi_\delta(z)=0$ for $|z|\geq\delta$ and $\alpha$ is given by
\begin{equation}
\label{alpha_short_1}
\alpha(y,\yh) = C_{\ss{f}}\; \ki{f}(y)\ki{f}(\yh) +
C_{\ss{m}}\; \ki{m}(y)\ki{m}(\yh) +
C_{\ss{i}}\left(\ki{f}(y)\ki{m}(\yh)+
\ki{m}(y)\ki{f}(\yh)\right).
\end{equation}

%Let $\e$ denote a length scale with $\e\rightarrow 0$. We extend $\chi_{{\tiny\mbox{fiber}}}$ by periodicity to $\mathbb{R}^3$ and
%rescale the pairwise force $f$ to $\Omega$ as follows
The short-range peridynamic force defined on $\Omega$  is given by
\begin{eqnarray}
\label{f_short}
f_{\ss{short}}^{\e}(\eta,\xi,x) =\frac{1}{\e^2}\,\alpha_{\e\delta}\left(\xe,\frac{x+\xi}{\e}\right) \;\frac{\xi\otimes\xi}{\vert\xi\vert^{3}} \eta.
\end{eqnarray}
\begin{figure}[t]%[ht]
\centering
\scalebox{0.2}{\includegraphics{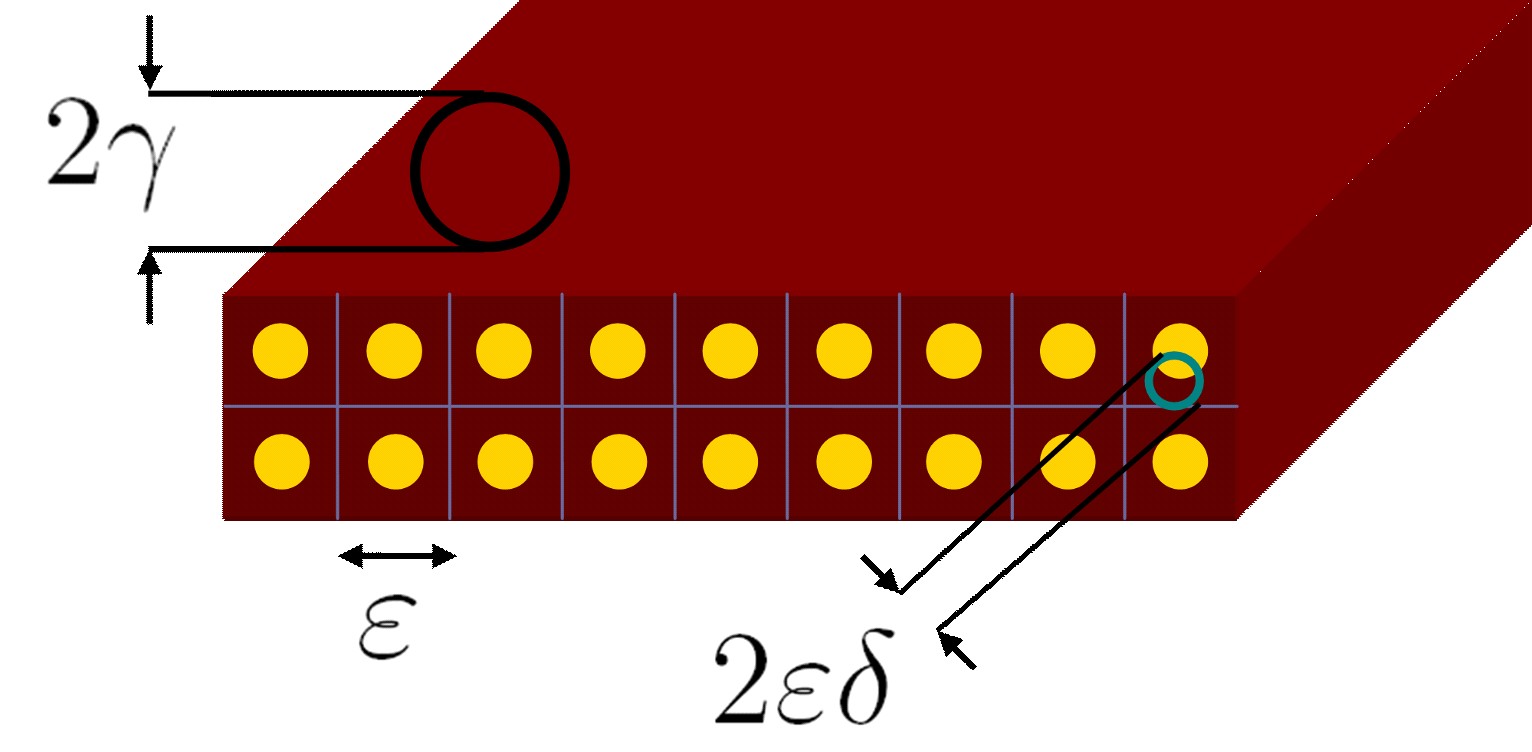}}
\caption{Long-range bonds (horizon $\gamma$) and  short-range bonds (horizon $\e \delta$).}
\label{Fi:LongPlusShort}
\end{figure}
For future reference we see from  (\ref{microhoriz}) and (\ref{alpha_short_1}) that $\alpha_{\e\delta}(\xe,\frac{\xh}{\e})$ is given by
%terms of $\chi_{\ss{f}}$ and $\chi_{\ss{m}}$ as follows.
\begin{equation}
\label{alpha_short_e}
\alpha_{\e\delta}\left(\frac{x}{\e},\frac{\xh}{\e}\right) = \chi_{\e\delta}(x-\xh)\left(C_{\ss{f}}\; \kie{f}(x)\kie{f}(\xh) +
C_{\ss{m}}\; \kie{m}(x)\kie{m}(\xh) +
C_{\ss{i}}\left(\kie{f}(x)\kie{m}(\xh)+
\kie{m}(x)\kie{f}(\xh)\right)\right),
\end{equation}
where $\kie{f}(x):=\ki{f}(\xe)$ and $\kie{m}(x):=\ki{m}(\xe)$. The oscillating density $\rho_\e$ for the heterogeneous medium is given by
$\rho_\e(x)=\rho(\frac{x}{\e})$.

The elastic displacement inside the heterogeneous body $\Omega$ is denoted by $u^\e(x,t)$ and 
the peridynamic equation of motion  for the heterogeneous medium is given by
\begin{equation}
\label{peridynamic1}
\renewcommand{\arraystretch}{1}
\begin{array}{lcl}
\rho_\e(x)\partial_{t}^2 u^{\e}(x,t) &=& \displaystyle\int_{H_  {\gamma}(x)\cap\Omega} f_{\ss{long}}(u^{\e}(\xh,t)-u^{\e}(x,t),\xi)\,d\xh \\
&&\\
& &+ \displaystyle\int_{H_  {\e \delta  }(x)\cap\Omega} f^\e_{\ss{short}}((u^{\e}(\xh,t)-u^{\e}(x,t)),\xi,x)\,d\xh \\
&&\\
& &+\displaystyle \,b^\e\left(x,t\right),\hbox{  for $x$ in $\Omega$.}
\end{array}
\end{equation}
The peridynamic equation is supplemented with initial conditions
\begin{eqnarray}
\label{peridynamic1ic1}
u^{\e}(x,0) &=&u_0^\e(x)\\
\label{peridynamic1ic2}
\partial_t u^{\e}(x,0)&=&v_0^\e(x).
\end{eqnarray}
Here the body force $b^\e(x,t)$ and initial conditions $u_0^\e(x)$, $v_0^\e(x)$ can depend upon $\e$. When these functions are bounded in $L^p(\Omega)^3$ for $p\geq 1$ it follows from the theory of semigroups that
there is a classic solution $u^\e(x,t)$ belonging to $C^2([0,T];\,L^p(\Omega)^3)$. This is discussed in the following section, see Remark \ref{L2initialandbodyforce}.

In what follows we will develop strong approximations for solutions $u^\e$ when the prescribed body forces and initial conditions are continuous at the coarse length scale but possess discontinuous oscillations over fine length scales. For this choice we look for a solution $u^\e(x,t)$ continuous in time but possibly discontinuous in the spacial variables and belonging to the Lebesgue space $L^p(\Omega)^3$ for $1\leq p<\infty$. In this paper we show that we can find  solutions $u^\e(x,t)$ and strong approximations of the form $u(x,x/\e,t)$ that both belong to $C^2([0,T];L^p(\Omega)^3)$, for a wide class of initial conditions and body forces. In order to describe this class of initial conditions and body forces we consider the space  $L^p_{\ss{per}}(Y;C(\overline{\Omega})^3)$ of functions $\psi(x,y)$ measurable with respect to $y$, $L^p$-integrable on $Y$ and $Y$-periodic in $y$, with values in the Banach space $C(\overline{\Omega})^3$ of continuous vector fields on $\overline{\Omega}$. 
Every element $\psi(x,y)$ of this space is a Caratheodory  function and hence  $\psi(x,\frac{x}{\e})$ is measurable on $\Omega$ and belongs to $L^p(\Omega)$. This kind of function space is well known in the context of two-scale convergence see,  \cite{Allaire}, and \cite{Nguetseng2}.
In what follows we will suppose $b(x,y,t)$ belongs to $C([0,T];\,L^p_{\ss{per}}(Y;C(\overline{\Omega})^3)$ and both
$u_0(x,y)$ and $v_0(x,y)$  belong to  $L^p_{\ss{per}}(Y;C(\overline{\Omega})^3)$. For this choice the initial conditions and body forces are given by $u^\e(x,0)=u_0(x,\frac{x}{\e})$, $\partial_t u^{\e}(x,0)=v_0(x,\frac{x}{\e})$, and  $b^\e(x,t)=b(x,\frac{x}{\e},t)$. The construction of a strong approximation for this  class of data is given in Theorem \ref{strongapproximation} of section \ref{ch_strongapproximation}.

It is important at this stage to point out that it is precisely the $\e^{-2}$ scaling of the bond force together with the scaling $\e\delta$ of the horizon that ultimately delivers the macroscopic equations
for $u^H$ given by \eqref{id_u_H}. In this context we expect other types of macroscopic equations to arise for different scalings of the bond force strength. Recent work for homogeneous media show that the classical equations of linear elasticity arise for bond force scaling on the order of $\e^{-4}$ and horizons with scaling $\e$, see 
\cite{Emmrich21}, \cite{Lehoucq1}, and \cite{Du1}.

When the initial conditions and body force are continuous functions and the density $\rho^\e$ and bond forces characterized by $\alpha(\xe,\frac{\xh}{\e})$ are also continuous then the solution $u^\e$ is continuous in space and belongs to $C^2([0,T];\,C(\overline{\Omega})^3)$; this is discussed in the next section. 

In forthcoming work we will focus on the development of strong approximations for initial conditions that are discontinuous with respect to coarse length scales. This will be carried out for heterogeneous peridynamic media characterized by oscillatory but continuous densities and bond forces. More generally one could contemplate strong approximations for more general combinations of bond forces and initial data.

 \begin{figure}[t]%[ht]
\centering
\scalebox{0.26}{\includegraphics{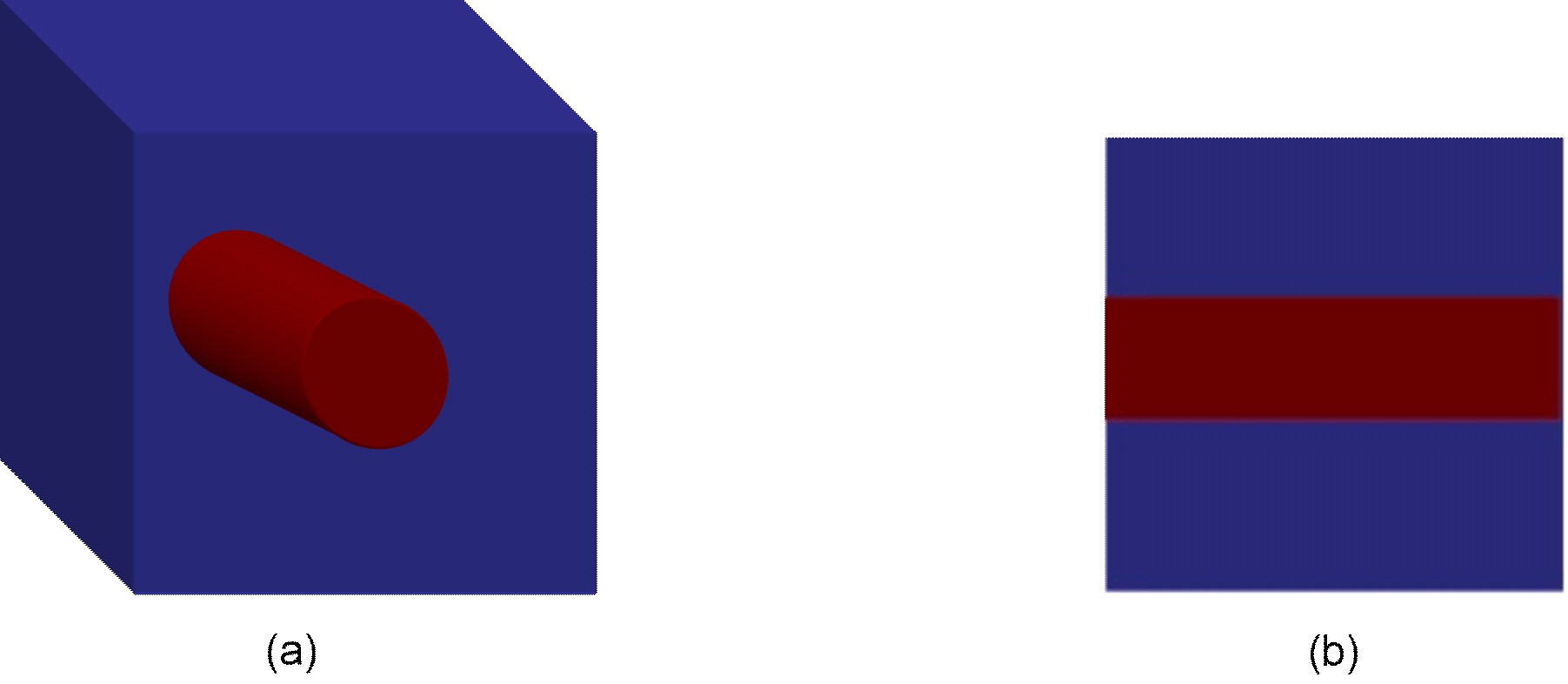}}
\caption{(a) Composite cube $Y$. (b) Cross-section of $Y$ along the fiber direction.}
\label{Fi:Y}
\end{figure}

\section{Peridynamic Formulation for Heterogeneous Media: A Well Posed Problem}
\label{ch_exist_and_uniq}
In this section, we make use of the semigroup theory of operators to show the existence and uniqueness of
solutions to (\ref{peridynamic1})-(\ref{peridynamic1ic2}).
For $v\in L^p(\Omega)^3$, with $1\leq p<\infty$, let
\begin{eqnarray}
\label{A_L1_0}
A^\e_{L,1}v(x)&=&\rho_\e^{-1}(x)\int_{H_{\gamma}(x)\cap\Omega} \lambda
\frac{(\xh-x)\otimes(\xh-x)}{\vert\xh-x\vert^{3}}\; v(\xh)\,d\xh,\\
\label{A_L2_0}
A^\e_{L,2}v(x)&=&\rho_\e^{-1}(x)\int_{H_{\gamma}(x)\cap\Omega} \lambda
\frac{(\xh-x)\otimes(\xh-x)}{\vert\xh-x\vert^{3}}\,d\xh\;v(x),\\
\label{Ae_S1_0}
A^{\e}_{S,1}v(x)&=& \rho_\e^{-1}(x)\int_{H_  {\e  \delta    }(x)\cap\Omega} \frac{1}{\e^{2}}\,\alpha\left(\xe,\frac{\xh}{\e}\right)
\frac{(\xh-x)\otimes(\xh-x)}{\vert\xh-x\vert^{3}}\; v(\xh)\,d\xh,\\ % \mbox{ and }
\label{Ae_S2_0}
A^{\e}_{S,2}v(x)&=& \rho_\e^{-1}(x)\int_{H_  {\e  \delta    }(x)\cap\Omega} \frac{1}{\e^{2}}\,\alpha\left(\xe,\frac{\xh}{\e}\right)
\frac{(\xh-x)\otimes(\xh-x)}{\vert\xh-x\vert^{3}}\,d\xh\;v(x).
\end{eqnarray}
Also we set
\begin{eqnarray}
\label{A_L}
A^\e_{L}&=&A^\e_{L,1}-A^\e_{L,2},\\
A^{\e}_{S}&=&A^{\e}_{S,1}-A^{\e}_{S,2},\\ % \mbox{ and }
\label{Ae}
A^{\e}&=&A^\e_{L} + A^{\e}_{S}.
\end{eqnarray}
Then by making the  identifications $u^{\e}(t)=u^{\e}(\cdot,t)$ and $b^{\e}(t)=b(\cdot,t)$, we can
write (\ref{peridynamic1})-(\ref{peridynamic1ic2}) as an operator equation in $L^p(\Omega)^3$
\begin{equation}
\label{ue_operator_eq}
%\nonumber
\displaystyle \left\{
\begin{array}{ccl}
\ddot{u}^{\e}(t) &=& A^{\e}u^{\e}(t)+\rho_\e^{-1}b^{\e}(t),\;\;\;\;\; t\in[ 0,T ]\\
u^{\e}(0)&=& u^{\e}_0,\\
\dot{u}^{\e}(0)&=& v^{\e}_0.
\end{array}
\right.
\end{equation}
or equivalently,  as an inhomogeneous Abstract Cauchy Problem in% defined on
 \\$L^p(\Omega)^3\times L^p(\Omega)^3$
\begin{equation}
\label{ue_ACP}
%\nonumber
\displaystyle \left\{
\begin{array}{ccl}
\dot{U}^{\e}(t) &=& \mathbb{A}^{\e}U^{\e}(t)+B^{\e}(t),\;\;\;\;\; t\in[ 0,T ]\\
U^{\e}(0)&=& U^{\e}_0.\\
\end{array}
\right.
\end{equation}
where
\begin{eqnarray*}
U^{\e}(t)=\left(%
\begin{array}{c}
  u^{\e}(t) \\
  \dot{u}^{\e}(t)\\
\end{array}%
\right),\,\,
U^{\e}_0=\left(%
\begin{array}{c}
  u^{\e}_0 \\
  v^{\e}_0\\
\end{array}%
\right),\,\,
B^{\e}(t)=\left(%
\begin{array}{c}
  0 \\
 \rho_\e^{-1} b^{\e}(t) \\
\end{array}%
\right), \mbox{ and }
\mathbb{A}^{\e} = \left(%
\begin{array}{cc}
  0 & I \\
  A^{\e} & 0 \\
\end{array}%
\right).
\end{eqnarray*}
Here $I$ denotes the identity map in $L^p(\Omega)^3$.
\begin{propo}
\label{existuniq_prop}
Let $1\leq p<\infty$ and assume that $b\in  C([0,T];\,L^p_{\ss{per}}(Y;C(\overline{\Omega})^3))$ and $U_0\in L^p_{\ss{per}}(Y;C(\overline{\Omega})^3)\times L^p_{\ss{per}}(Y;C(\overline{\Omega})^3)$. Then
\begin{enumerate}
%\item[(a)] The operators $A^{\e}: L^s(\Omega)^3\rightarrow L^s(\Omega)^3$ and
%$\mathbb{A}^{\e}: L^s(\Omega)^3\times L^s(\Omega)^3\rightarrow L^s(\Omega)^3\times L^s(\Omega)^3$ are  linear and
%uniformly bounded  in $\e$.
\item[(a)] The operators $A^{\e}$ and $\mathbb{A}^{\e}$ are  linear and bounded on $L^p(\Omega)^3$ and
$L^p(\Omega)^3\times L^p(\Omega)^3$, respectively. Moreover, the bounds are uniform in $\e$.

\item[(b)] Equation (\ref{ue_ACP}) has a unique classical solution $U^{\e}$ in
$C^1([0,T];\,L^p(\Omega)^3\times L^p(\Omega)^3)$ which is given by
\begin{equation}
\label{Ue_explicit}
U^{\e}(t) = e^{t \mathbb{A}^{\e}}U^{\e}_0 + \int_0^t e^{(t-\tau) \mathbb{A}^{\e}}B^{\e}(\tau)\,d\tau, \,\,\, t\in [0,T],
\end{equation}
where
\begin{equation}
\label{semigroup}
e^{t \mathbb{A}^{\e}} = \sum_{n=0}^{\infty} \frac{t^n}{n!}\; (\mathbb{A}^{\e})^n.
\end{equation}
Moreover, equation (\ref{ue_operator_eq}) has a unique classical solution $u^{\e}\in C^2([0,T];\,L^p(\Omega)^3)$ which is given by
\begin{subequations}
\label{ue_explicit}
\begin{eqnarray}
\nonumber
u^{\e}(t) &=& \cosh{\left(t\sqrt{A^{\e}}\right)} u^{\e}_0 + \sqrt{A^{\e}}^{\;-1}\sinh{\left(t\sqrt{A^{\e}}\right)} v^{\e}_0\\
& &+ \sqrt{A^{\e}}^{\;-1}\!\int_0^t \sinh{\left((t-\tau)\sqrt{A^{\e}}\right)}b^{\e}(\tau)\,d\tau
\end{eqnarray}
with the notation
\begin{eqnarray}
\cosh{\left(t\sqrt{A^{\e}}\right)}  :=  \sum_{n=0}^{\infty} \frac{t^{2n}}{(2n)!}\; (A^{\e})^n\\
\sqrt{A^{\e}}^{\;-1}\sinh{\left(t\sqrt{A^{\e}}\right)} := \sum_{n=0}^{\infty} \frac{t^{2n+1}}{(2n+1)!}\; (A^{\e})^n
\end{eqnarray}
\end{subequations}

%\begin{eqnarray}
%\label{ue_explicit}
%\nonumber
%u^{\e}(t) &=&  \sum_{n=0}^{\infty} \frac{t^{2n}}{(2n)!}\; (A^{\e})^n u^{\e}_0+
%\sum_{n=0}^{\infty} \frac{t^{2n+1}}{(2n+1)!}\; (A^{\e})^n v^{\e}_0\\
%\nonumber
%\\
%& & +\int_0^t  \sum_{n=0}^{\infty} \frac{(t-\tau)^{2n+1}}{(2n+1)!}\; (A^{\e})^n b^{\e}(\tau)\,d\tau.
%\end{eqnarray}
\item[(c)] The   sequences $(u^{\e})_{\e>0}$, $(\dot{u}^{\e})_{\e>0}$, and $(\ddot{u}^{\e})_{\e>0}$ are
 bounded in $L^{\infty}([0,T];\,L^p(\Omega)^3)$.
% \item[(d)] There is a function $m\in L^1(\Omega)^3$ such that
% \[
% |u^{\e}(x,t)|\leq m(x),
% \]
% for all $x\in \Omega$ and $t\in[0,T]$.
 % can show this directly from the representation of u^{\e} or U^{\e} using:
 % |A^{\e}v(x)| is bounded and by simple induction step |A^{\e}^n v(x)| is bounded.
 % This is needed to justify some uses of LDCT
 % will use another way in the two-scale chapter
\label{existuniq_prop_c}
\end{enumerate}
\end{propo}
\begin{rem_no_numbers}
\label{L2initialandbodyforce}
The hypothesis of Proposition \ref{existuniq_prop} can be relaxed by assuming that the sequences of initial
conditions $\left(u_0^{\e}\right),\left(v_0^{\e}\right)$, are bounded 
in $L^p(\Omega)^3$ and $\left(b^{\e}(\cdot,t)\right)$ is uniformly bounded in $L^p(\Omega)^3$ for $t\in[0,T]$. This is proved following the same steps given in the proof of Proposition \ref{existuniq_prop} presented below.  
\end{rem_no_numbers}
\begin{proof}
Part (a). It is clear that the operators $A^{\e}_{S,1}$, $A^{\e}_{S,2}$, $A^{\e}_{L,1}$, and $A^{\e}_{L,2}$ are linear. So we begin the proof by showing that
$A^{\e}_{S,1}$ and $A^{\e}_{S,2}$ are uniformly bounded sequences of operators on $L^p(\Omega)^3$ for $1\leq p<\infty$. We introduce the indicator function $\chi_{\scriptscriptstyle{\Omega}}(x)$ taking the value one for $x$ inside $\Omega$ and zero for $x$ outside $\Omega$ and let $v$ denote a generic vector field belonging to $L^p(\Omega)^3$.
Then
by the change of variables $\xh=x+\e z$ in (\ref{Ae_S1_0}) we obtain
\begin{eqnarray}
\label{Ae_S1_1}
A^{\e}_{S,1}v(x)&=& \rho_\e^{-1}\int_{H_  {\delta    }(0)}\chi_{\scriptscriptstyle{\Omega}}(x+\e z) \alpha\left(\xe,\xe+z\right)
\frac{z\otimes z}{\vert z\vert^{3}}\; v(x+\e z)\,dz.
\end{eqnarray}
Applying Minkowski's inequality gives
\begin{eqnarray}
\Vert A^{\e}_{S,1}v(x)\Vert_{L^p(\Omega)^3}\leq\int_{H_  {\delta    }(0)}\left(\int_\Omega\chi_{\scriptscriptstyle{\Omega}}(x+\e z) \rho^{-1}(\frac{x}{\e})|\alpha\left(\xe,\xe+z\right)
\frac{z\otimes z}{\vert z\vert^{3}}\; v(x+\e z)|^p\,dx\right)^{1/p}\,dz.
\label{minkowski}
\end{eqnarray}
Let $\displaystyle\overline{\alpha}=\max_{y,y'\in Y}\rho^{-1}(y)\alpha(y,y')$ and we see that
\begin{eqnarray}
\label{Ae_S1_5}
\nonumber
\Vert A^{\e}_{S,1}v(x)\Vert_{L^p(\Omega)^3}&\leq& \overline{\alpha}\int_{H_  {\delta    }(0)}
\frac{1}{\vert z\vert}\; \left(\int_\Omega\chi_{\scriptscriptstyle{\Omega}}(x+\e z)\vert v(x+\e z)\vert^p\,dx\right)^{1/p}\,dz\\
&\leq& M_{S}
\Vert v\Vert_{L^p(\Omega)^3},
\end{eqnarray}
where  $M_{S}$ is independent of $\e$ and given by
\begin{eqnarray}
\label{M_S}
M_{S}=\overline{\alpha}\;\left(\int_{H_  {\delta    }(0)}  \frac{1}{\vert z\vert}\,dz\right)
=\overline{\alpha}\frac{2\pi\delta^2}{3},
\end{eqnarray}
which shows that the operators are $A^{\e}_{S,1}$ is uniformly bounded with respect to $\e$.
Similarly, $A^{\e}_{S,2}$ can be written as
\begin{eqnarray}
\label{Ae_S2_1}
A^{\e}_{S,2}v(x)&=& \int_{H_  {\delta    }(0)}\chi_{\scriptscriptstyle{\Omega}}(x+\e z)\rho^{-1}(\frac{x}{\e}) \alpha\left(\xe,\xe+z\right)
\frac{z\otimes z}{\vert z\vert^{3}}\,dz\,\; v(x).
\end{eqnarray}
Thus
\begin{eqnarray}
\label{Ae_S2_2}
\nonumber
\vert A^{\e}_{S,2}v(x)\vert &\leq& M_{S}\; \vert v(x)\vert,
\end{eqnarray}
from which the boundedness of $A^{\e}_{S,2}$ immediately follows. Combining these results shows that
$A_S^{\e}$, which is given by $A^{\e}_{S,1}-A^{\e}_{S,2}$,  is a sequence of uniformly bounded operators on $L^p(\Omega)^3$.

Next we show that  the linear operators $A_L^\e=A^{\e}_{L,1}-A^{\e}_{L,2}$  are a sequence of  uniformly bounded operators on
$L^p(\Omega)^3$. 
Changing variables $\hat{x}=x+\xi$ and applying Minkowski's inequality gives
\begin{eqnarray}
\Vert \,A^{\e}_{L,1}v\,\Vert_{L^p(\Omega)^3}&\leq&\int_{H_  {\gamma   }(0)}\left(\int_\Omega \chi_{\scriptscriptstyle{\Omega}}(x+\xi)\rho^{-1}(\frac{x}{\e})|\lambda
\frac{\xi\otimes \xi}{\vert \xi\vert^{3}}\; v(x+\xi)|^p\,dx\right)^{1/p}\,d\xi\nonumber\\
&\leq& M_L \Vert \,v\,\Vert_{L^p(\Omega)^3}
\label{A_L1_1}
\end{eqnarray}
where $M_{L}$ is given by
\begin{eqnarray}
\label{M_L}
M_{L}=\max_{y\in Y}\{\rho^{-1}(y)\}\times\lambda\frac{2\pi\gamma^2}{3},
\end{eqnarray}
and it follows that the operator $A^{\e}_{L,1}$ is bounded in $L^p(\Omega)^3$.
%%%%%%%%%%%%%%%%
The boundedness of $A^{\e}_{L,2}$, which is given by (\ref{A_L2_0}), follows immediately from its definition.
Therefore $A^{\e}_L$ is uniformly bounded on $L^p(\Omega)^3$ with respect to $\e$.

Since $A^{\e}=A^{\e}_L+A^{\e}_S$, we conclude that
\begin{eqnarray}
\label{Ae_1}
\Vert A^{\e}v\Vert_{L^p(\Omega)^3}\leq M \,\Vert v\Vert_{L^p(\Omega)^3},
\end{eqnarray}
for a positive constant $M$ which is independent of $\e$.
The operator $\mathbb{A}^{\e}$ is clearly linear, thus it remains to show that this operator is
uniformly bounded on $L^p(\Omega)^3\times L^p(\Omega)^3$. To see this,
we let $(v,w)\in L^p(\Omega)^3\times L^p(\Omega)^3$. The norm in this Banach space is given by
\[
\Vert (v,w)\Vert_{L^p(\Omega)^3\times L^p(\Omega)^3} = \Vert v\Vert_{L^p(\Omega)^3}+ \Vert w\Vert_{L^p(\Omega)^3}.
\]
We note that
\begin{eqnarray*}
\mathbb{A}^{\e}
{\small \left(%
\begin{array}{c}
  v \\
  w \\
\end{array}%
\right)
=
\left(%
\begin{array}{cc}
  0 & I \\
  A^{\e} & 0 \\
\end{array}%
\right)
\left(%
\begin{array}{c}
  v \\
  w \\
\end{array}%
\right)
=
\left(%
\begin{array}{c}
  w \\
  A^{\e} v \\
\end{array}%
\right).
}
\end{eqnarray*}
Thus we obtain
\begin{eqnarray}
\label{bigAe_1}
\nonumber
\Vert \mathbb{A}^{\e}(v,w)\Vert_{L^p(\Omega)^3\times L^p(\Omega)^3} &=& \Vert w\Vert_{L^p(\Omega)^3}+ \Vert A^{\e} v\Vert_{L^p(\Omega)^3}\\
&\leq& \Vert w\Vert_{L^p(\Omega)^3}+ \Vert A^{\e}\Vert \,\Vert v\Vert_{L^p(\Omega)^3}.
\end{eqnarray}
From (\ref{bigAe_1}) it follows that
\begin{eqnarray}
\label{bigAe_2}
\Vert \mathbb{A}^{\e}(v,w)\Vert_{L^p(\Omega)^3\times L^p(\Omega)^3}
\leq M \Vert (v,w)\Vert_{L^p(\Omega)^3\times L^p(\Omega)^3},
\end{eqnarray}
for some positive constant $M$ completing the argument.

Part (b). We have seen from Part (a) that $\mathbb{A}^{\e}$ is a bounded linear operator on the Banach space
$L^p(\Omega)^3\times L^p(\Omega)^3$. Also, since $b^{\e}$ is in $C([0,T];\,L^p(\Omega)^3)$, it follows that
$B^{\e}=(0,b^{\e})$ is in $C([0,T];\,L^p(\Omega)^3\times L^p(\Omega)^3)$. From these facts, it follows from the theory
of semigroups, see for example, \cite{Pazy,Engel}.
\begin{enumerate}
\item The operator $\mathbb{A}^{\e}$ generates a uniformly continuous semigroup $\{e^{t\mathbb{A}^{\e}}\}_{t\ge 0}$
on $L^p(\Omega)^3\times L^p(\Omega)^3$, where $e^{t\mathbb{A}^{\e}}$ is given by (\ref{semigroup}).
\item The inhomogeneous Abstract Cauchy Problem (\ref{ue_ACP}) has a unique classical solution
$U^{\e}\in C^1([0,T];\,L^p(\Omega)^3\times L^p(\Omega)^3)$ which is given by (\ref{Ue_explicit}).
\label{item_2}
\end{enumerate}
It immediately follows from (\ref{item_2}) that  the second order inhomogeneous Abstract Cauchy Problem (\ref{ue_operator_eq}) has
a unique classical solution
$u^{\e}\in C^2([0,T];\,L^p(\Omega)^3)$ and formula (\ref{ue_explicit}) follows immediately from (\ref{semigroup}).

Part (c). We recall that
\begin{eqnarray*}
u^{\e}_0(x)&:=& u_0(x,\xe)\\
v^{\e}_0(x) &:=& v_0(x,\xe)
\end{eqnarray*}
where
$u_0(x,y)$, $v_0(x,y)$ are in $L^p_{\ss{per}}(Y; C(\overline{\Omega})^3)$.
%It follows that (see \cite))
%\begin{eqnarray*}
%\Vert u^{\e}_0\Vert_{L^p(\Omega)^3}&\leq& \Vert u^0\Vert_{L^p_{\ss{per}}(Y;C(\overline{\Omega})^3)}:=
%\left(\int_{Y} \sup_{x\in\Omega}\vert u_0(x,y)\vert^p\,dy\right)^{1/p}, \\
%\Vert v^{\e}_0\Vert_{L^p(\Omega)^3} &\leq& \Vert v^0\Vert_{L^p_{\ss{per}}(Y;C(\overline{\Omega})^3)}:=
%\left(\int_{Y} \sup_{x\in\Omega}\vert v_0(x,y)\vert^p\,dy\right)^{1/p}.
%\end{eqnarray*}
%Thus $u_0^{\e}$ and $v_0^{\e}$ are uniformly bounded in $L^p(\Omega)^3$, which implies that $U^{\e}_0$ is 
%uniformly bounded in $L^p(\Omega)^3\times L^p(\Omega)^3$. 
%Bacim commented the above March 10,2010
We surround $\Omega$ by a cube of integer side length $L$ and extend $u_0(x,y)$ to the cube by setting $u_0(x,y)=0$ for $x$ outside $\Omega$ and for every $y$ in $Y$. We note that the extended $u_0(x,\xe)$ is $\e=\frac{1}{n}$ periodic in the second variable and shift the cube so that it is commensurate with the periods.
The period cells of side length $\e$ are denoted by $\e Y_i$ and the cube is given by their union $\cup_{i}\e Y_i$
where the index $i$ ranges from $1$ to $L^3n^3$. Since we have extended $u_0(x,y)$ so that it vanishes when $x$ lies outside $\Omega$ one can write
\begin{eqnarray}
\label{estn}
\Vert u^{\e}_0\Vert_{L^p(\Omega)^3}&=&\left(\int_{\cup_i \e Y_i}|u_0(x,\xe)|^p\,dx\right)^{1/p}.
\end{eqnarray}
Hence
\begin{eqnarray}
\Vert u^{\e}_0\Vert_{L^p(\Omega)^3}&\leq& \left(\int_{\cup_i \e Y_i} \sup_{x'\in\Omega}\vert u_0(x',\xe)\vert^p\,dx\right)^{1/p}\nonumber\\
&=&\left(\sum_{i=1}^{L^3n^3}\int_{\e Y_i} \sup_{x'\in\Omega}\vert u_0(x',\xe)\vert^p\,dx\right)^{1/p}\nonumber\\
&=&L^{3/p}\Vert u_0\Vert_{L^p_{\ss{per}}(Y;C(\overline{\Omega})^3)}.
\label{initialestu_0}
\end{eqnarray}
Here the last inequality follows from the change of variables $y=\xe$.
Thus $u_0^{\e}$ is uniformly bounded in $L^p(\Omega)^3$. Similarly $v_0^\e$ is uniformly bounded which implies that $U^{\e}_0$ is uniformly bounded
in $L^p(\Omega)^3\times L^p(\Omega)^3$. 
The same considerations show that for $t\in [0,T]$, that
$b^{\e}(t)$ is uniformly bounded in $L^p(\Omega)^3$. Since $b^{\e}(t)$ is continuous in $t$, it follows that $b^{\e}$
is uniformly bounded in $C([0,T];\,L^p(\Omega)^3)$, which implies that $B^{\e}$ is uniformly bounded
in $C([0,T];\,L^p(\Omega)^3\times L^p(\Omega)^3)$.

Next we note that
\begin{eqnarray}
\label{bigAe_ineq}
\nonumber
\Vert e^{t \mathbb{A}^{\e}}\Vert &\leq&  e^{t \Vert\mathbb{A}^{\e}\Vert}\\
&\leq&  e^{t M},
\end{eqnarray}
where in the last inequality we have used the fact that $\mathbb{A}^{\e}$ is uniformly bounded. Taking the
 norm in both sides of (\ref{Ue_explicit}) and by using (\ref{bigAe_ineq}), we obtain
\begin{eqnarray}
\label{Ue_ineq}
%\nonumber
\Vert U^{\e}(t)\Vert_{L^p(\Omega)^3\times L^p(\Omega)^3} &\leq& M_1 e^{t M} +
\int_0^t e^{(t-\tau) M} M_2\,d\tau,
\end{eqnarray}
for some positive numbers $M_1$, $M_2$, and $M$. This implies that $U^{\e}$ is uniformly bounded
in $L^{\infty}([0,T];\,L^p(\Omega)^3\times L^p(\Omega)^3)$. Therefore the sequences $(u^{\e})_{\e>0}$ and  $(\dot{u}^{\e})_{\e>0}$
are bounded in $L^{\infty}([0,T];\,L^p(\Omega)^3)$. Finally, it follows from equation (\ref{ue_operator_eq}) that
the sequence $(\ddot{u}^{\e})_{\e>0}$ is bounded in $L^{\infty}([0,T];\,L^p(\Omega)^3)$, completing the proof.
\end{proof}

It is easily seen that for continuous initial conditions and body forces that the peridynamic solution $u^\e$ is also continuous in space provided that the bond forces and densities are continuous. To fix ideas we ``smooth out'' the  characteristic functions $\chi_{\ss{f}}$ and $\chi_{\ss{m}}$ by mollification. Indeed given any infinitely differential function $\zeta$ with compact support on $\Omega$ we fix $\beta$ such that $0<\beta<\delta$ and form $\zeta^\beta(x)=\beta^{-3}\zeta(\frac{x}{\beta})$. The mollified characteristic functions are given by $\chi^\beta_{\ss{f}}(x)=(\zeta^\beta*\chi_{\ss{f}})(x)$ and $\chi^\beta_{\ss{m}}(x)=(\zeta^\beta*\chi_{\ss{m}})(x)$.
The replacement of $\chi_{\ss{f}}$ and $\chi_{\ss{m}}$ by their mollified counter parts in \eqref{denstyfluct} and \eqref{alpha_short_1}  delivers a short range bond force $f_{\ss{short}}^\e(\eta,\xi,x)$ and density  $\rho^\e(x)$ that are continuous in $x$. For this case it is easy to see that $A^\e_{S,1}$, $A^\e_{S,2}$, $A_{L,1}$, and $A_{L,2}$ are  linear operators mapping  $C(\overline{\Omega})^3$ into itself. A straight forward application of H\"older's inequality shows that $A^\e_{S,1}$, $A^\e_{S,2}$, $A_{L,1}$, and $A_{L,2}$ are bounded
and that the operator norms of $A^\e_{S,1}$, $A^\e_{S,2}$ are uniformly bounded with respect to $\e$. To fix ideas we choose $u_0$ and $v_0$ in $C(\overline{\Omega})$ and for $b$ in $C^1([0,T];C(\Omega))$  and proceeding as before we find that the solution $u^\epsilon$ of the peridynamic initial value problem exists is unique and
belongs to $C^2([0,T];C(\overline{\Omega})^3)$.

\section{Strong Approximation by Two-Scale Functions }
\label{ch_two_scale_cong}

The aim of this  section is to build an approximation of $u^\e(x,t)$ when the period $\e$ of the microstructure is small.
In what follows we show how to systematically identify a function $u(x,y,t)$ that is oscillatory with respect to a new ``fast'' spatial variable $y$ that when rescaled $y=\frac{x}{\e}$ delivers a strong approximation to $u^\e(x,t)$, i.e.,
\begin{eqnarray}
\label{strongapprox}
\lim_{\e\rightarrow 0}\Vert u^\e(x,t)-u(x,\frac{x}{\e},t)
\Vert_{L^p(\Omega)}=0.
\end{eqnarray}
It is  shown that the desired function $u(x,y,t)$ is the ``two-scale'' limit of the sequence $\{u^\e(x,t)\}$ for $\e\rightarrow 0$. After periodically extending $u(x,y,t)$  in the $y$ variable we find that it satisfies  the two-scale peridynamic initial-value problem  given in theorem \ref{two-scale-limit-eqn-thm}. In the subsequent sections we apply this fact to show that $u(x,\frac{x}{\e},t)$ provides a strong approximation to $u^\e(x,t)$ when $\e$ is sufficiently small.

\subsection{Two-Scale Convergence}
\label{sec_two_scale}
To expedite the presentation we list the following useful function spaces 
\begin{eqnarray*}
\mathcal{K} &=& \{\psi\in  C^{\infty}_c(\mathbb{R}^3\times Y),\,\, \psi(x,y) \mbox{ is $Y$-periodic in } y\},\\
\mathcal{J} &=& \{\psi\in  C^{\infty}_c(\mathbb{R}^3\times Y\times \mathbb{R}^{+}),\,\, \psi(x,y,t) \mbox{ is $Y$-periodic in } y\},\\
\mathcal{L}_p &=& \{w\in  C([0,T];\,L^p_{\ss{per}}(Y;C(\overline{\Omega})^3) \},\\
\mathcal{Q}_p &=& \{w\in  C^{2}([0,T];\,L^p_{\ss{per}}(Y;C(\overline{\Omega})^3)\}.
\end{eqnarray*}
and introduce the definition of two-scale convergence.
Let $p$ and $p'$ be two real numbers such that $1\leq p<\infty$ and $1/p+1/p'=1$.
\begin{defn2}[Two-scale convergence \cite{Nguetseng,Allaire}]
\label{def_two_scale}
A sequence $(v^{\e})$ of functions in $L^p(\Omega)$, is said to two-scale converge to a limit
$v\in L^p(\Omega\times Y)$ if, as $\e\rightarrow 0$
\begin{equation}
\int_{\Omega} v^{\e}(x) \psi\left(x,\frac{x}{\e}\right)\,dx\rightarrow \int_{\Omega\times Y} v(x,y) \psi(x,y)\, dx dy
\end{equation}
for all $\psi\in L^{p'}(\Omega;\,C_{per}(Y))$. We will often use $v^{\e}\stackrel{2}{\rightharpoonup} v$ to denote that $(v^{\e})$
two-scale converges to $v$.
\end{defn2}
If the sequence $(v^{\e})$ is bounded in $L^p(\Omega)$ then $L^{p'}(\Omega;\,C_{per}(Y))$ can be replaced by
%$\mathcal{D}(\Omega;\,C^{\infty}_{per}(Y))$
$\mathcal{K}$ in Definition (\ref{def_two_scale}) (see \cite{Nguetseng2}).
%The following is a well-known fact in two-scale convergence.
For time-dependent problems one slightly modifies the above two-scale convergence
to allow for homogenization with a parameter, see \cite{Showalter,E}. Here the parameter is denoted by $t$.
%These changes do not affect the proofs in \cite{} in any essential way. cite Clarkson Showalter and W.E.
\begin{defn2}
\label{def_two_scale-time}
A bounded sequence $(v^{\e})$ of functions in $L^p(\Omega\times (0,T))$, is said to two-scale converge to a limit
$v\in L^p(\Omega\times Y\times (0,T))$ if, as $\e\rightarrow 0$
\begin{equation}
\int_{\Omega\times (0,T)} v^{\e}(x,t) \psi\left(x,\frac{x}{\e},t\right)\,dx dt\rightarrow
\int_{\Omega\times Y\times (0,T)} v(x,y,t) \psi(x,y,t)\, dx dy dt
\end{equation}
for all $\psi\in \mathcal{J}$.
\end{defn2}

Definition \ref{def_two_scale} is motivated by the following compactness result of Nguetseng, see \cite{Nguetseng} and Allaire \cite{Allaire}.
\begin{theorem}%[Nguetseng]
\label{thm_Nguetseng}
Let $(v^{\e})$ be a bounded sequence  in $L^p(\Omega)$. Then there exists a subsequence and a function $v\in L^p(\Omega\times Y)$
such that the subsequence two-scale converges to $v$.
\end{theorem}

A similar two-scale  compactness holds for time dependent problems and is stated in the following theorem.
%For the problems studied in this work,
%it is more convenient to have the following variation of Nguetseng's
%theorem, which has been presented in \cite{} for the case $p=2$.
\begin{theorem}
\label{two_scale_thm_with_time}
Let $(v^{\e})$ be a bounded sequence  in $L^p(\Omega \times (0,T))^3$. Then there exists a subsequence and a function
$v\in L^p(\Omega\times Y\times (0,T))^3$
such that the subsequence two-scale converges to $v$.
\end{theorem}
The proof of compactness for the time dependent case is essentially the same as the proof of Theorem \ref{thm_Nguetseng}. 
A slight variation of
Theorem \ref{two_scale_thm_with_time} can be found  in \cite{E} and \cite{Showalter}.
For future reference we recall the following well known results on two-scale convergence that can be found in \cite{Nguetseng2}.
\begin{propo}
\label{two_scale_time_weak_convg}
Let $(v^{\e})$ be a bounded sequence  in $L^p(\Omega \times (0,T))^3$ that two-scale converges to $v\in L^p(\Omega\times Y \times (0,T))^3$.
Then as $\e\rightarrow 0$
\begin{equation*}
v^{\e}\rightarrow \int_{Y} v(x,y,t)\,dy \;\;\mbox{ weakly in } L^p(\Omega\times (0,T))^3.
\end{equation*}
\end{propo}
\begin{propo}
\label{strongvstwoscale}
If  $v^\e(x)$ converges to $v(x)$ in $L^p(\Omega)^3$   then its two-scale limit is $v$.
\end{propo}
Last we state two-scale convergence theorems for test functions.
\begin{propo}
\label{smooth_two_scale}
If $\psi(x,y)$ belongs to  $\mathcal{K}$ or $L_{\ss{per}}^p(Y;C(\overline\Omega)^3)$ then $\psi(x,\frac{x}{\e})$ two-scale converges to $\psi(x,y)$ and 
\begin{eqnarray}
\label{test}
\lim_{\e\rightarrow 0}\Vert\psi(x,\frac{x}{\e})\Vert_{L^p(\Omega)}^p=\int_{\Omega\times Y}|\psi(x,y)|^p\,dx\,dy.
\end{eqnarray}
Moreover given any bounded sequence $v^\e$ in $L^p(\Omega)^3$ two-scale converging to $v$ then 
\begin{eqnarray}
\lim_{\e\rightarrow 0}\int_\Omega v^\e(x)\psi(x,\frac{x}{\e})\,dx=\int_{\Omega\times Y}v(x,y)\psi(x,y)\,dxdy
\label{twotest}
\end{eqnarray}
for every test function $\psi$ belonging to $L_{\ss{per}}^p(Y;C(\overline\Omega)^3)$.

Similarly if $\psi(x,y,t)$ belongs to  $\mathcal{J}$ or $\mathcal{L}_p$ then $\psi(x,\frac{x}{\e},t)$ two-scale converges to $\psi(x,y,t)$ and 
\begin{eqnarray}
\label{testt}
\lim_{\e\rightarrow 0}\Vert\psi(x,\frac{x}{\e},t)\Vert_{L^p(\Omega\times(0,T))^3}^p=\int_{\Omega\times Y\times(0,T)}|\psi(x,y,t)|^p\,dx\,dy\,dt.
\end{eqnarray}
Moreover given any bounded sequence $v^\e$ in $L^p(\Omega\times (0,T))^3$ two-scale converging to $v$ then 
\begin{eqnarray}
\lim_{\e\rightarrow 0}\int_{\Omega \times (0,T)}v^\e(x,t)\psi(x,\frac{x}{\e},t)\,dxdt=\int_{\Omega\times Y\times(0,T)}v(x,y,t)\psi(x,y,t)\,dxdydt
\label{twotestt}
\end{eqnarray}
for every test function $\psi$ belonging to $\mathcal{L}_p$.
\end{propo}

\subsection{The Two-Scale Limit Equation}
In this section, we use two-scale convergence to identify the limit of the solution $u^\e(x,t)$ of (\ref{peridynamic1})-(\ref{peridynamic1ic2}) for initial data $u_0^\e=u_0(x,\frac{x}{\e})$, $v_0=v_0(x,\frac{x}{\e})$ and body force $b^\e(x,\frac{x}{\e},t)$ with $u_0$ and $v_0$ in $L^p_{per}(Y;C(\overline{\Omega})^3)$ and $b\in\mathcal{L}_p$.
For $v\in L^p(\Omega)^3$, with $\frac{3}{2}< p<\infty$, let
\begin{eqnarray}
\label{K_L1_0}
K_{L,1}v(x)&=&\int_{H_{\gamma}(x)\cap\Omega} \lambda
\frac{(\xh-x)\otimes(\xh-x)}{\vert\xh-x\vert^{3}}\; v(\xh)\,d\xh,\\
\label{K_L2_0}
K_{L,2}v(x)&=&\int_{H_{\gamma}(x)\cap\Omega} \lambda
\frac{(\xh-x)\otimes(\xh-x)}{\vert\xh-x\vert^{3}}\,d\xh\;v(x),\\
\label{Ke_S1_0}
K^{\e}_{S,1}v(x)&=& \int_{H_  {\e  \delta    }(x)\cap\Omega} \frac{1}{\e^{2}}\,\alpha\left(\xe,\frac{\xh}{\e}\right)
\frac{(\xh-x)\otimes(\xh-x)}{\vert\xh-x\vert^{3}}\; v(\xh)\,d\xh,\\ % \mbox{ and }
\label{Ke_S2_0}
K^{\e}_{S,2}v(x)&=& \int_{H_  {\e  \delta    }(x)\cap\Omega} \frac{1}{\e^{2}}\,\alpha\left(\xe,\frac{\xh}{\e}\right)
\frac{(\xh-x)\otimes(\xh-x)}{\vert\xh-x\vert^{3}}\,d\xh\;v(x).
\end{eqnarray}
Set $K_L=K_{L,1}-K_{L,2}$ and $K^{\e}_S=K^{\e}_{S,1}-K^{\e}_{S,2}$ and the peridynamic equation \eqref{peridynamic1} is written
\begin{eqnarray}
\label{pd}
\rho(\frac{x}{\e})\partial^2_{t}u^\e(x,t)=\left(K_L +K^\e_{S}\right)u^{\e}(x,t)+b\left(x,\frac{x}{\e},t\right)
\end{eqnarray}
We start by noting that
the loading force and initial data are in $\mathcal{L}_p$ and $L_{\ss{per}}^p(Y;C(\overline{\Omega})^3)$ respectively and from Proposition \ref{smooth_two_scale} satisfy the following
\begin{subequations}
\label{four-case-conds}
\begin{eqnarray}
b\left(x,\xe,t\right)&\stackrel{2}{\rightharpoonup}& b(x,y,t),\\
u_0\left(x,\xe\right)&\stackrel{2}{\rightharpoonup}&u_0(x,y),\\
v_0\left(x,\xe\right)&\stackrel{2}{\rightharpoonup}&v_0(x,y).
\end{eqnarray}
\end{subequations}
We note that from
 Proposition \ref{existuniq_prop}(c) and Theorem \ref{two_scale_thm_with_time} it follows that, up to some subsequences,
$\twoscale{u^{\e}}{u}$, $\twoscale{\dot{u}^{\e}}{u^{*}}$, and $\twoscale{\ddot{u}^{\e}}{u^{**}}$,
where $u$, $u^{*}$, and $u^{**}$ are in
$L^p(\Omega\times Y\times [0,T])^3$. 
We shall see later that
$u(x,y,t)$ is uniquely determined by an initial value problem. Therefore $u$ is independent of the subsequence, and the whole
sequence $(u^{\e})$ two-scale converges to $u$.

We start by extending the function $u(x,y,t)$ in the $y$ variable from $Y$ to $\mathbb{R}^3$ as a $Y$-periodic function. 
The next task is to identify the dynamics of the periodically extended $u(x,y,t)$. 
We multiply both sides of \eqref{pd} by a test function
$\psi(x,\frac{x}{\e},t)$, where $\psi(x,y,t)$ is $Y$-periodic in $y$ and is such that
$\psi\in C^{\infty}_{c}(\mathbb{R}^3\times Y\times \mathbb{R})^3$,
and integrate over $\Omega\times \mathbb{R}^{+}$
\begin{eqnarray}
\nonumber
\int_{\Omega\times \mathbb{R}^{+}} &&\partial^{2}_{t}u^{\e}(x,t)\cdot \psi\left(x,\frac{x}{\e},t\right)\rho(\frac{x}{\e})\,dx dt\\
\nonumber
=\int_{\Omega\times \mathbb{R}^{+}} &&\left((K_{L}+K^{\e}_{S})u^{\e}(x,t) +b\left(x,\xe,t\right)\right)\cdot \psi\left(x,\frac{x}{\e},t\right)\,dx dt
\end{eqnarray}
After integrating by parts twice, we obtain
\begin{equation}
\renewcommand{\arraystretch}{1.5}
\begin{array}{ll}
\displaystyle
\nonumber
\int_{\Omega\times \mathbb{R}^{+}} u^{\e}(x,t)\cdot\partial^{2}_{t} \psi\left(x,\frac{x}{\e},t\right)\rho(\frac{x}{\e})\,dx dt -
\int_{\Omega} \partial_{t}u^{\e}(x,0)\cdot \psi\left(x,\frac{x}{\e},0\right)\rho(\frac{x}{\e})\,dx &\\
\label{ue_and_two_scale}
\displaystyle
\nonumber
+\int_{\Omega} u^{\e}(x,0)\cdot \partial_{t}\psi\left(x,\frac{x}{\e},0\right)\rho(\frac{x}{\e})\,dx&\\
\displaystyle
= \int_{\Omega\times \mathbb{R}^{+}} \left((K_{L}+K^{\e}_{S})u^{\e}(x,t)+b\left(x,\xe,t\right)\right)\cdot \psi\left(x,\frac{x}{\e},t\right)\,dx dt&
\end{array}
\end{equation}

Passing to the  $\e\rightarrow 0$ limit we obtain
\begin{eqnarray}
%&&\\
\nonumber
&&\int_{\Omega\times Y\times \mathbb{R}^{+}} u(x,y,t)\cdot\partial^{2}_{t} \psi(x,y,t)\rho(y)\,dx dy dt -
\int_{\Omega\times Y} v_0(x,y)\cdot \psi(x,y,0)\rho(y)\,  dx dy \\
\label{limit_ue_and_two_scale}
\nonumber
&&+\int_{\Omega\times Y} u_0(x,y)\cdot\partial_{t} \psi(x,y,0)\rho(y)\,  dx dy \\
\nonumber
&&= \lim_{\e\rightarrow 0}\int_{\Omega\times \mathbb{R}^{+}} (K_{L}+K^{\e}_{S})u^{\e}(x,t)\cdot \psi\left(x,\frac{x}{\e},t\right)\,dx dt\\
&&+\int_{\Omega\times Y\times \mathbb{R}^{+}}b(x,y,t)\cdot \psi(x,y,t)\,dx dy dt
\end{eqnarray}

We will use  the following lemma to compute the limit on the right hand side of (\ref{limit_ue_and_two_scale}).
\begin{lemma}
\label{two_scale_lemma}
Let  $w$ be in $L^{p}_{\ss{per}}(Y;\,C(\overline{\Omega})^3)$ with $\frac{3}{2}<p<\infty$, and define
\begin{eqnarray*}
B_L w(x,y)&=&\int_{H_{\gamma}(x)\cap\Omega} \lambda
\frac{(\xh-x)\otimes(\xh-x)}{\vert\xh-x\vert^{3}}\left(\int_{Y}w(\xh,y')\,dy'-w(x,y)\right)\,d\xh,\\
B_S w(x,y) &=&\int_{H_{\delta}(y)} \alpha(y,\yh)
\frac{(\yh-y)\otimes(\yh-y)}{\vert\yh-y\vert^{3}}\left(w(x,\yh)-w(x,y)\right)\,d\yh.
\end{eqnarray*}
Then as $\e \rightarrow 0$,
\begin{enumerate}
\item[(a)] $\displaystyle\twoscale{K_{L}u^{\e}(x,t)}{B_L u(x,y,t)}$.\\
Moreover, the operator $\rho^{-1}B_L$ is  linear and bounded on $L^{p}_{\ss{per}}(Y;\,C(\overline{\Omega})^3)$.

\item[(b)] $\displaystyle\twoscale{K^{\e}_{S}u^{\e}(x,t)}{B_S u(x,y,t)}$.\\
Moreover, the operator $\rho^{-1}B_S$ is linear and bounded on $L^{p}_{\ss{per}}(Y;\,C(\overline{\Omega})^3)$.
\end{enumerate}
\end{lemma}
\noindent The proof of this lemma is provided at the end of this subsection.

\begin{rem_no_numbers}
Results similar to Lemma \ref{two_scale_lemma} can be proven for other function spaces as well. The  space $L^{p}_{\ss{per}}(Y;\,C(\overline{\Omega})^3)$ in the statement 
of this lemma can, for example, be replaced with the function space $L^p_{\ss{per}}(Y;\,L^p(\Omega)^3)$ or by the function space $L^p(\Omega;\,C_{\ss{per}}(Y)^3)$, 
where $\frac{3}{2}<p<\infty$ in each of 
these spaces.
\end{rem_no_numbers}

Application of  Lemma (\ref{two_scale_lemma}) gives
\begin{eqnarray*}
&&\lim_{\e\rightarrow 0}\int_{\Omega\times \mathbb{R}^{+}} (K_{L}+K^{\e}_{S})u^{\e}(x,t)\cdot \psi\left(x,\frac{x}{\e},t\right)\,dx dt \\
&&=\int_{\Omega\times Y\times\mathbb{R}^{+}} (B_{L}+B_{S})u(x,y,t)\cdot \psi(x,y,t)\,dx dy dt.
\end{eqnarray*}
Thus (\ref{limit_ue_and_two_scale}) becomes
\begin{eqnarray}
\nonumber
&&\int_{\Omega\times Y\times \mathbb{R}^{+}} u(x,y,t)\cdot\partial^{2}_{t} \psi(x,y,t)\rho(y)\,dx dy dt -
\int_{\Omega\times Y} v_0(x,y)\cdot \psi(x,y,0)\rho(y)\,  dx dy\\
\nonumber
\label{two_scale_limit_integral_form}
&&+\int_{\Omega\times Y} u_0(x,y)\cdot\partial_{t} \psi(x,y,0)\rho(y)\,  dx dy\\
&&= \int_{\Omega\times Y\times\mathbb{R}^{+}} \left((B_{L}+B_{S})u(x,y,t)+b(x,y,t)\right)\cdot \psi(x,y,t)\,dx dy dt
\end{eqnarray}
We shall see from Lemma \ref{g-lemmma}, provided before the end of this subsection, that $u$ has two classical partial
derivatives with respect to $t$, for almost every $t$, and the initial conditions supplementing (\ref{two_scale_limit_integral_form}) are given by
\begin{equation}
\label{two_scale_ic}
u(x,y,0)=u_0(x,y),\;\; \partial_{t}u(x,y,0)=v_0(x,y).
\end{equation}
Thus by integrating by parts twice, equation (\ref{two_scale_limit_integral_form}) becomes
\begin{eqnarray}
\label{two_scale_limit_integral_form2}
\nonumber
&&\int_{\Omega\times Y\times \mathbb{R}^{+}} \rho(y)\partial^{2}_{t}u(x,y,t)\cdot \psi(x,y,t)\,dx dy dt\\
&&=\int_{\Omega\times Y\times\mathbb{R}^{+}} \left((B_{L}+B_{S})u(x,y,t)+b(x,y,t)\right)\cdot \psi(x,y,t)\,dx dy dt
\end{eqnarray}
Since this is true for any function $\psi\in C^{\infty}_{c}(\mathbb{R}^3\times Y\times \mathbb{R})^3$ for
which $\psi(x,y,t)$ is $Y$-periodic in $y$, we obtain that for almost every $x,y,$ and $t$
\begin{eqnarray}
\label{two_scale_limit_eq}
\partial^{2}_{t}u(x,y,t) = \rho^{-1}(y)B u(x,y,t) + \rho^{-1}b(x,y,t),
\end{eqnarray}
where $B=B_L+B_S$. It follows from Lemma \ref{two_scale_lemma} that $\rho^{-1}B$ is a bounded linear operator on $L^p_{\ss{per}}(Y;\, C(\overline{\Omega})^3)$, with $\frac{3}{2}<p<\infty$.
Therefore 
the initial value problem given by (\ref{two_scale_limit_eq}) and (\ref{two_scale_ic}),
interpreted as a second-order inhomogeneous abstract Cauchy problem defined on $L^p_{\ss{per}}(Y;\, C(\overline{\Omega})^3)$, with body force in $\mathcal{L}_p, \, \frac{3}{2}<p<\infty$. From the theory of semigroups  \cite{Pazy,Engel}
it follows that this problem has a unique solution $u(x,y,t)\in \mathcal{Q}_p, \, \frac{3}{2}<p<\infty$.

The following summarizes the results of this subsection.
\begin{theorem}
\label{two-scale-limit-eqn-thm}
Let $(u^{\e})$ be the sequence of solutions of (\ref{peridynamic1})-(\ref{peridynamic1ic2}) with initial data $u_0^\e=u_0(x,\frac{x}{\e})$, $v_0=v_0(x,\frac{x}{\e})$ and body force $b^\e(x,\frac{x}{\e},t)$ with $u_0$ and $v_0$ in $L^p_{per}(Y;C(\overline{\Omega})^3)$ and $b\in\mathcal{L}_p$.
Then \\ $\twoscale{u^{\e}}{u}$ and the periodic extension of $u(x,y,t)$ in the $y$ variable from $Y$ to $\mathbb{R}^3$ also denoted by  $u$ belongs to $\mathcal{Q}_p$, with
$\frac{3}{2}<p<\infty$,  and is the unique solution of
\begin{equation}
\label{two_scale_limit_full_form}
\renewcommand{\arraystretch}{1}
\begin{array}{lcl}
\rho(y)\partial_{t}^2 u(x,y,t) &=& \displaystyle\int_{H_{\gamma}(x)\cap\Omega} \!\!\!\!\!\lambda
\frac{(\xh-x)\otimes(\xh-x)}{\vert\xh-x\vert^{3}}\left(\int_{Y}u(\xh,y',t)\,dy'-u(x,y,t)\right)\,d\xh\\
&&\\
& &+ \displaystyle\int_{H_{\delta}(y)} \alpha(y,\yh)
\frac{(\yh-y)\otimes(\yh-y)}{\vert\yh-y\vert^{3}}\left(u(x,\yh,t)-u(x,y,t)\right)\,d\yh \\
&&\\
& &+ \,b(x,y,t),%\nonumber
\end{array}
\end{equation}
supplemented with initial conditions
\begin{eqnarray}
\label{two_scale_limit_full_formic1}
u(x,y,0) &=& u_0(x,y),\\
\label{two_scale_limit_full_formic2}
\partial_t u(x,y,0)&=&v_0(x,y).
\end{eqnarray}
\end{theorem}

We conclude this section by showing that $u$ is twice differentiable with respect to time and proving Lemma \ref{two_scale_lemma}.
\begin{lemma}
\label{g-lemmma}
Let $t\in[0,T]$ and define
\begin{equation}
\label{g}
g(x,y,t) = \int_{0}^{t}\int_{0}^{\tau}u^{**}(x,y,l)\,dl d\tau + t u^{*}(x,y,0)+u(x,y,0).
\end{equation}
Then $g$ is in $L^p(\Omega\times Y\times (0,T))^3$, twice differentiable with respect to $t$ almost everywhere, and satisfies
\begin{enumerate}
\item[(a)] For almost every $x,y$, and $t$, $g(x,y,t)=u(x,y,t)$, $\partial_{t}g(x,y,t)=u^{*}(x,y,t) $,\\
 and $\partial^{2}_{t}g(x,y,t)=u^{**}(x,y,t)$.
\item[(b)] For almost every $x$ and $y$
\begin{eqnarray*}
\label{two_scale_ic1}
g(x,y,0)=u(x,y,0)=u_0(x,y),\\
\label{two_scale_ic2}
\partial_{t}g(x,y,0) = u^{*}(x,y,0)= v_0(x,y).
\end{eqnarray*}
\end{enumerate}
\end{lemma}
\begin{proof}
Part (a). Let $\psi_1(x,y)$ be in $C^{\infty}_c(\Omega\times Y)^3$ and $Y$-periodic in $y$, and let $\phi$ be in $C^{\infty}_c(\mathbb{R}^+)$.
Then by using integration by parts, we see that
\begin{eqnarray*}
\label{ut_eq_u*_1}
\nonumber
\int_{\Omega\times \mathbb{R}^+} \partial_t u^{\e}(x,t)\cdot \psi_1\left(x,\xe\right)\phi(t)\,dx dt=
-\int_{\Omega\times \mathbb{R}^+}  u^{\e}(x,t)\cdot \psi_1\left(x,\xe\right)\dot\phi(t)\,dx dt.
\end{eqnarray*}
Sending $\e$ to $0$ and  using the fact that, up to a subsequence, $\twoscale{\partial_t u^{\e}}{u^{*}}$, we obtain
\begin{eqnarray*}
\int_{\Omega\times Y \times\mathbb{R}^+} u^{*}(x,y,t)\cdot \psi_1\left(x,y\right)\phi(t)\,dx dy dt\\
=-\int_{\Omega\times Y\times\mathbb{R}^+}  u(x,y,t)\cdot \psi_1\left(x,y\right)\dot\phi(t)\,dx dy dt.
\end{eqnarray*}
Since this holds for every $\psi_1$ we conclude that
\begin{eqnarray}
\label{ut_eq_u*}
\int_{\mathbb{R}^+} u^{*}(x,y,t) \phi(t)\,dt=
-\int_{\mathbb{R}^+}  u(x,y,t) \dot\phi(t)\,dt,
\end{eqnarray}
for almost every $x$ and $y$ and for every $\phi\in C^{\infty}_c(\mathbb{R}^+)$. Similarly, by using the fact that, up to a subsequence,
$\twoscale{\partial^2_t u^{\e}}{u^{**}}$, we see that
%Similarly and by replacing $\partial_t u^{\e}(x,t)$ with $\partial^{2}_t u^{\e}(x,t)$ in the left hand side of (\ref{ut_eq_u*_1}),
%then integrating by parts twice, and then sending $\e$ to $0$, we see that for almost every $x$ and $y$
\begin{eqnarray}
\label{utt_eq_u**}
\int_{\mathbb{R}^+} u^{**}(x,y,t) \phi(t)\, dt=
\int_{\mathbb{R}^+}  u(x,y,t)\ddot\phi(t)\, dt,
\end{eqnarray}
for almost every $x$ and $y$ and for every $\phi\in C^{\infty}_c(\mathbb{R}^+)$.
 We note that from
(\ref{g}) it is easy to see that $g$ is twice differentiable in $t$ almost everywhere and satisfies
\begin{eqnarray}
\label{g_t}
%g(x,y,0)&=&u(x,y,0),\\
\partial_t g(x,y,t) &=& \int_{0}^{t}u^{**}(x,y,\tau)\, d\tau + u^{*}(x,y,0),\\
%\partial_t g(x,y,0)&=&u^{*}(x,y,0),\\
\label{g_tt}
\partial^2_t g(x,y,t) &=& u^{**}(x,y,t).
\end{eqnarray}
We will use these facts together with (\ref{ut_eq_u*}) and (\ref{utt_eq_u**}) to show that  $\partial_{t}g=u^{*}$ almost everywhere
and $g=u$ almost everywhere.

For $\phi\in C^{\infty}_c(\mathbb{R}^+)$, we integrate by parts using (\ref{g_tt}) and (\ref{utt_eq_u**}) to find that
\begin{eqnarray*}
\label{g_t_eq_u*}
\int_{\mathbb{R}^+}\partial_t g(x,y,t) \dot\phi(t)\,dt&=&  \int_{\mathbb{R}^+} u^{*}(x,y,t) \dot\phi(t)\,dt.
\end{eqnarray*}
Thus we obtain
\begin{eqnarray}
\label{g_t_eq_u*_1}
\int_{\mathbb{R}^+}\left(\partial_t g(x,y,t) - u^{*}(x,y,t)\right) \dot\phi(t)\,dt=0,
\end{eqnarray}
for every $\phi\in C^{\infty}_c(\mathbb{R}^+)$. Since  $\partial_t g(x,y,0)=u^{*}(x,y,0)$, we conclude from (\ref{g_t_eq_u*_1})
that $\partial_{t}g(x,y,t)=u^{*}(x,y,t)$ almost everywhere. Finally it easily follows from \eqref{ut_eq_u*}) that
\begin{eqnarray}
\label{g_eq_u_1}
\int_{\mathbb{R}^+}\left( g(x,y,t) - u(x,y,t)\right) \dot\phi(t)\,dt=0,
\end{eqnarray}
for every $\phi\in C^{\infty}_c(\mathbb{R}^+)$. Since  $ g(x,y,0)=u(x,y,0)$, we conclude from (\ref{g_eq_u_1})
that $g(x,y,t)=u(x,y,t)$ almost everywhere, completing the proof of Part (a).

Part (b). Let $\psi(x,y,t)$ be in $C^{\infty}_c(\Omega\times Y\times \mathbb{R})^3$ and $Y$-periodic in $y$.
Then on integrating by parts, we see that
\begin{eqnarray*}
\label{g_at_0_1}
\nonumber
\int_{\Omega\times \mathbb{R}^+} \partial_t u^{\e}(x,t)\cdot \psi\left(x,\xe,t\right)\,dx dt&=&
-\int_{\Omega\times \mathbb{R}^+}  u^{\e}(x,t)\cdot \partial_t\psi\left(x,\xe,t\right)\,dx dt \\
&&-\int_{\Omega}u^{\e}(x,0)\cdot \psi\left(x,\xe,0\right)\,dx.
\end{eqnarray*}
Sending $\e$ to $0$, we obtain
\begin{eqnarray}
\label{g_at_0_2}
\nonumber
\int_{\Omega\times Y \times\mathbb{R}^+}\!\! u^{*}(x,y,t)\cdot \psi\left(x,y,t\right)\,dx dy dt
&=&-\int_{\Omega\times Y\times\mathbb{R}^+} \!\! u(x,y,t)\cdot \partial_t\psi\left(x,y,t\right)\,dx dy dt\\
\nonumber
&& -\int_{\Omega\times Y}u_0(x,y)\cdot \psi\left(x,y,0\right)\,dx dy.\\
\end{eqnarray}
On the other hand, from Part (a), we see that
\begin{eqnarray}
\label{g_at_0_3}
\nonumber
\int_{\Omega\times Y \times\mathbb{R}^+}\!\! u^{*}(x,y,t)\cdot \psi\left(x,y,t\right)\,dx dy dt
&=& - \int_{\Omega\times Y \times\mathbb{R}^+} \!\!u(x,y,t) \cdot \partial_t\psi\left(x,y,t\right)\,dx dy dt\\
\nonumber
&&-\int_{\Omega\times Y} \!\!u(x,y,0) \cdot \psi\left(x,y,0\right)\,dx dy.\\
\end{eqnarray}
From (\ref{g_at_0_2}) and (\ref{g_at_0_3}) we obtain that
\begin{eqnarray*}
\label{g_at_0_4}
\int_{\Omega\times Y}\left(u_0(x,y) - u(x,y,0)\right)\cdot \psi\left(x,y,0\right)\,dx dy=0,
\end{eqnarray*}
for every $\psi$. Therefore
\[
u(x,y,0)=u_0(x,y),
\]
almost everywhere. Similarly we can show that
\[
\partial_t u(x,y,0)=v_0(x,y),
\]
almost everywhere, completing the proof of Part (b).
\end{proof}

\begin{proof}[Proof of Lemma \ref{two_scale_lemma}]
Part (a). We compute  the two-scale limits of $K_{L,1} u^{\e}$ and $K_{L,2} u^{\e}$ to show that as $\e \rightarrow 0$,
\begin{eqnarray}
\label{AL_2scale}
\twoscale{K_L u^{\e}(x,t)}{B_L u(x,y,t)}.
\end{eqnarray}
Let $\psi\in C^{\infty}_c(\mathbb{R}^{3}\times Y)^3$ such that $\psi(x,y)$ is $Y$-periodic in $y$,
and $\phi\in C^{\infty}_c(\mathbb{R}^{+})$. Then  from
the definition of $K_{L,1}$, equation (\ref{K_L1_0}), we see that
\begin{eqnarray}
\label{AL1we_1}
\nonumber
&&\int_{\Omega\times \mathbb{R}^{+}}K_{L,1}u^{\e}(x,t)\cdot \psi\left(x,\xe\right)\phi(t)\,dx dt\\
\nonumber
&& = \int_{\Omega\times \mathbb{R}^{+}}\int_{H_{\gamma}(x)\cap\Omega} \lambda
\frac{(\xh-x)\otimes(\xh-x)}{\vert\xh-x\vert^{3}}\; u^{\e}(\xh,t)\,d\xh \cdot \psi\left(x,\xe\right)\phi(t)\,dx dt,\\
\end{eqnarray}
Since $\twoscale{u^{\e}(x,t)}{u(x,y,t)}$, we obtain using Proposition \ref{two_scale_time_weak_convg} that, as
$\e\rightarrow 0$,
\begin{eqnarray}
\label{ue_2scale_t}
u^{\e}\rightarrow \int_Y u(x,y,t)\,dy\,\,\,\mbox{weakly in } L^p(\Omega\times (0,T))^3.
\end{eqnarray}
It follows from (\ref{ue_2scale_t}) that, for fixed $x$,
\begin{eqnarray}
\label{AL1we_2}
&&\lim_{\e\rightarrow 0}
\int_{\mathbb{R}^{+}}\int_{H_{\gamma}(x)}\chi_\Omega(\xh) \lambda
\frac{(\xh-x)\otimes(\xh-x)}{\vert\xh-x\vert^{3}}\; u^{\e}(\xh,t) \phi(t)\,d\xh dt\\
\nonumber
&& =\int_{\mathbb{R}^{+}}\int_{H_{\gamma}(x)} \chi_\Omega(\xh)\lambda
\frac{(\xh-x)\otimes(\xh-x)}{\vert\xh-x\vert^{3}} \left(\int_Y u(\xh,y',t)\,dy'\right)\phi(t)\,d\xh dt.
\end{eqnarray}
Here $\chi_\Omega$ is the characteristic function of $\Omega$, taking value $1$ for $\xh$ in $\Omega$ and zero outside.
Applying H\"older's inequality for $\frac{1}{p}+\frac{1}{p'}=1$ gives
\begin{eqnarray}
\label{justify_Lebesgues_1}
\nonumber
&&\left\vert\int_{H_{\gamma}(x)} \!\!\!\!\chi_\Omega(\xh)\lambda
\frac{(\xh-x)\otimes(\xh-x)}{\vert\xh-x\vert^{3}}\; u^{\e}(\xh,t)\,d\xh\right\vert\\
\nonumber
 && \leq
\lambda\;\left(\int_{H_  {\delta    }(x)}\chi_\Omega(\xh)  \frac{1}{\vert \xh - x\vert^{p'}}\,d\xh\right)^{1/p'}
\left(\int_{H_  {\delta    }(x)} \vert u^{\e}(\xh,t )\vert^{p} \,d\xh\right)^{1/p}\\
&& \leq
\lambda \;\left(\int_{H_  {\delta    }(x)}  \frac{1}{\vert \xh - x\vert^{p'}}\,d\xh\right)^{1/p'}
\Vert u^{\e}\Vert_{L^{\infty}([0,T];\,L^p(\Omega)^3)},\hbox{ for almost every $t\in [0,T]$}.
\end{eqnarray}
We note that the integral on the right hand side of the last inequality is finite for $p'<3$. From Proposition \ref{existuniq_prop}, $\Vert u^{\e}\Vert_{L^{\infty}([0,T];\,L^p(\Omega)^3)}$ is bounded. Thus
from (\ref{AL1we_2}), and (\ref{justify_Lebesgues_1}) and by using Lebesgue's dominated convergence
theorem, we conclude that the convergence of the sequence of functions in (\ref{AL1we_2}) is not only point-wise in $x$ convergence
but also strong
in $L^p(\Omega)^3$, with $\frac{3}{2}<p<\infty$. Therefore from Proposition \ref{strongvstwoscale} and (\ref{AL1we_2}) it follows that the limit of (\ref{AL1we_1}) as
$\e\rightarrow 0$ is given by
\begin{eqnarray}
\label{AL1ue_limit}
\nonumber
&&\!\!\!\!\!\!\!\!\!\lim_{\e\rightarrow 0} \int_{\Omega\times \mathbb{R}^{+}}K_{L,1}u^{\e}(x,t)\cdot \psi\left(x,\xe\right)\phi(t)\,dx dt\\
\nonumber
&&\!\!\!\!\!\!\!\!\! = \int_{\Omega\times \mathbb{R}^{+}\times Y}\,B_{L,1}u(x,y,t)\cdot \psi\left(x,y\right)\phi(t)\,dxdtdy,\\
\end{eqnarray}
where
\begin{eqnarray}
\label{B_L_1}
B_{L,1}u(x,y,t)=\int_{H_{\gamma}(x)\cap\Omega} \!\!\!\lambda
\frac{(\xh-x)\otimes(\xh-x)}{\vert\xh-x\vert^{3}}\left(\int_Y u(\xh,y',t)\,dy'\right)\,d\xh
\end{eqnarray}
depends only on $(x,t)$ and is constant in $y$.
Next we evaluate the two-scale limit of $K_{L,2} u^{\e}$. We recall from (\ref{A_L2_0}) that
\begin{eqnarray}
\label{AL2ue_1}
K_{L,2} u^{\e}(x,t) = \int_{H_{\gamma}(x)\cap\Omega} \lambda
\frac{(\xh-x)\otimes(\xh-x)}{\vert\xh-x\vert^{3}}\,d\xh \; u^{\e}(x,t),
\end{eqnarray}
from which immediately follows that as $\e \rightarrow 0$,
\begin{eqnarray}
\label{AL2ue_limit}
\twoscale{K_{L,2} u^{\e}}{\int_{H_{\gamma}(x)\cap\Omega} \lambda
\frac{(\xh-x)\otimes(\xh-x)}{\vert\xh-x\vert^{3}}\,d\xh \; u(x,y,t)}\equiv B_{L,2}u(x,y,t).
\end{eqnarray}
The result (\ref{AL_2scale}) follows on combining equations  (\ref{AL1ue_limit}) and (\ref{AL2ue_limit})
and writing $B_L=B_{L,1}-B_{L,2}$.
It is evident that $\rho^{-1}B_L$ is a linear operator on the Banach space  $L^p_{\ss{per}}(Y;\,C(\overline{\Omega})^3)$.   
To show boundedness we show that $\rho^{-1}B_{L,1}$ and $\rho^{-1}B_{L,2}$ are bounded operators on $L^p_{\ss{per}}(Y;\,C(\overline{\Omega})^3)$.
For $w$ in $L^p_{\ss{per}}(Y;\,C(\overline{\Omega})^3)$ we write $\xh=x+\xi$ and 
\begin{eqnarray}
\label{B_L_1_Bound}
&&\Vert \rho^{-1}B_{L,1}w\Vert_{L^p_{\ss{per}}(Y;\,C(\overline{\Omega})^3))}\\
&&=\left(\int_{Y}\left(\rho^{-1}(y)\sup_{x\in\Omega}\left\vert\int_{H_\gamma(0)}\chi_\Omega(x+\xi)\lambda\frac{\xi\otimes\xi}{|\xi|}\int_Y\,w(x+\xi,y')dy'd\xi\right\vert\right)^pdy\right)^{1/p}\nonumber\\
&&\leq\left(\int_{Y}\left(\rho^{-1}(y)\int_{H_\gamma(0)}\frac{\lambda}{|\xi|}\int_Y\,\sup_{x\in\Omega}|\chi_\Omega(x+\xi)w(x+\xi,y')|dy'd\xi\right)^pdy\right)^{1/p}\nonumber\\
&&\leq\lambda\int_{H_\gamma(0)}|\xi|^{-1}d\xi\left(\int_Y\,\int_Y\left(\rho^{-1}(y)\sup_{\xh\in\Omega}|w(\xh,y')|dy\right)^pdy'\right)^{1/p}\nonumber\\
&&\leq\lambda\int_{H_\gamma(0)}|\xi|^{-1}d\xi\Vert\rho^{-1}\Vert_{L^p(Y)}\Vert w\Vert_{L^p_{\ss{per}}(Y;\,C(\overline{\Omega})^3)},\nonumber
\end{eqnarray}
where the second inequality follows from Minkowski's inequality and it follows that $\rho^{-1}B_{L,1}$ is bounded. It is evident from the definition of $B_{L,2}$
that $\rho^{-1}B_{L,2}$ is a bounded operator on $L^p_{\ss{per}}(Y;\,C(\overline{\Omega})^3)$.

Part (b). Since $K^{\e}_S=K^{\e}_{S,1}-K^{\e}_{S,2}$, we will compute  the two-scale limits of $K^{\e}_{S,1} u^{\e}$
and $K^{\e}_{S,2} u^{\e}$, to show that as $\e \rightarrow 0$,
\begin{eqnarray}
\label{AS_2scale}
\twoscale{K^{\e}_S u^{\e}(x,t)}{B_S u(x,y,t)}.
\end{eqnarray}
Let $\psi(x,y,t)=\psi_2(x) \psi_1(y) \phi(t)$,
where $\psi_2\in C^{\infty}_c(\mathbb{R}^{3})$, $\psi_1\in C^{\infty}_{per}(Y)^3$, and $\phi\in C^{\infty}_c(\mathbb{R}^{+})$.
Then  by using  (\ref{Ae_S1_1}), replacing $v(x)$ with $u^{\e}(x,t)$, we have
\begin{eqnarray}
\label{AS1ue_1}
\nonumber
&&\int_{\Omega\times \mathbb{R}^{+}}K^{\e}_{S,1}u^{\e}(x,t)\cdot \psi\left(x,\xe,t\right)\,dx dt\\
\nonumber
&& = \int_{\Omega\times \mathbb{R}^{+}}\int_{H_  {\delta    }(0)} \chi_\Omega(x+\e z)\alpha\left(\xe,\xe+ z\right)
\frac{z\otimes z}{\vert z\vert^{3}}\; u^{\e}(x+\e z,t)\,dz \cdot \psi\left(x,\xe,t\right)\,dx dt,\\
\end{eqnarray}
where $\chi_{\Omega}$ denotes the indicator function of $\Omega$.
 Thus after a change in the order of integration in the right hand side of equation (\ref{AS1ue_1}), we see that
\begin{eqnarray}
\label{AS1ue_2}
\nonumber
&&\!\!\!\!\!\int_{\Omega\times \mathbb{R}^{+}}K^{\e}_{S,1}u^{\e}(x,t)\cdot \psi\left(x,\xe,t\right)\,dx dt\\
\nonumber
&&\!\!\!\!\! = \int_{H_  {\delta    }(0)}\frac{1}{|z|^3}\int_{\Omega\times \mathbb{R}^{+}} \chi_\Omega(x+\e z)
\alpha\left(\xe,\xe+ z\right)u^{\e}(x+\e z,t)\!\cdot\! z \;\psi_1\left(\xe\right)\!\cdot\! z\;\psi_2(x)\phi(t)\, dx dt dz.\\
\end{eqnarray}
Now we focus on evaluating the limit as $\e\rightarrow 0$ of the inner integral in (\ref{AS1ue_2}).
By  the change of
variables $r=x+\e z$ we obtain
\begin{eqnarray}
\label{AS1ue_fix_1}
\nonumber
&&\!\!\!\!\!\!\!\!\!
\int_{\Omega\times \mathbb{R}^{+}} \chi_\Omega(x+\e z)
\alpha\left(\xe,\xe+ z\right)
u^{\e}(x+\e z,t)\!\cdot\! z \;\psi_1\left(\xe\right)\!\cdot\! z\;\psi_2(x)\phi(t)\,dx dt\\
\nonumber
&&\!\!\!\!\!\!\!\!\!
=   \int_{\mathbb{R}^3\times \mathbb{R}^{+}} \chi_\Omega(r)\chi_{\Omega}(r-\e z)\,\alpha\left(\frac{r}{\e}-z,\frac{r}{\e}\right)
u^{\e}(r,t)\!\cdot\! z \;\psi_1\left(\frac{r}{\e}-z\right)\!\cdot\! z\;\psi_2(r-\e z)\phi(t)\,dr dt\\
&&\\
\nonumber
&&\!\!\!\!\!\!\!\!\!
:= a^{\e}(z),
\end{eqnarray}
We will show that for $z\in H_{\delta }(0)$,
\begin{eqnarray}
\label{AS1ue_3}
\nonumber
&&\lim_{\e\rightarrow 0}
%\int_{\mathbb{R}^3} \chi_{\Omega}(r-\e z)\,\ki{f}\left(\frac{r}{\e}-z\right) \ki{f}\left(\frac{r}{\e}\right)
%u^{\e}(r,t)\!\cdot\! z \;\psi_1\left(\frac{r}{\e}-z\right)\!\cdot\! z\;\psi_2(r-\e z)\,dr\\
a^{\e}(z)
= \int_{\Omega\times Y\times\mathbb{R}^{+}} \alpha\left(y-z,y\right)
u(r,y,t)\!\cdot\! z \;\psi_1\left(y-z\right)\!\cdot\! z\;\psi_2(r) \phi(t)\, dr dy dt.\\
\end{eqnarray}
To see this, we approximate $\chi_{\Omega}$ by smooth functions $\zeta_n$ such that as $n \rightarrow \infty$,
$\zeta_n\rightarrow \chi_{\Omega}$ in $L_{loc}^{p'}(\mathbb{R}^3)$, with $1/p+1/p'=1$. Then by adding and subtracting $\zeta_n(r-\e z)$
$\chi_{\Omega}(r-\e z)$ in (\ref{AS1ue_fix_1}), we see that
\begin{eqnarray}
\label{AS1ue_ae}
 \!\!\!\!\!\!\!\!\!\!\!\!\!\!\!\!\!\!
 a^{\e}(z)&=& a^{n,\e}_1(z)+a^{n,\e}_2(z),\\
\nonumber
\!\!\!\!\!\!\!\!\!\!\!\!\!\!\!\!\!\!
\mbox{where,}&&\\
\label{AS1ue_fix_2}
\nonumber
\!\!\!\!\!\!\!\!\!\!\!\!\!\!\!\!\!\!
a^{n,\e}_1(z) &:=&\int_{\mathbb{R}^3\times\mathbb{R}^{+}}\!\!\! \chi_\Omega(r)\left(\chi_{\Omega}(r-\e z)-\zeta_n(r-\e z)\right)\times\\
\!\!\!\!\!\!\!\!\!\!\!\!\!\!\!\!\!\!
&&\alpha\left(\frac{r}{\e}-z,\frac{r}{\e}\right)
u^{\e}(r,t)\!\cdot\! z \;\psi_1\left(\frac{r}{\e}-z\right)\!\cdot\! z\;\psi_2(r-\e z) \phi(t)\,dr dt,\\
\nonumber
\!\!\!\!\!\!\!\!\!\!\!\!\!\!\!\!\!\!
a^{n,\e}_2(z) &:=&
\int_{\mathbb{R}^3\times\mathbb{R}^{+}} \chi_\Omega(r)\zeta_n(r-\e z)\times\\
\!\!\!\!\!\!\!\!\!\!\!\!\!\!\!
&&\alpha\left(\frac{r}{\e}-z,\frac{r}{\e}\right)
u^{\e}(r,t)\!\cdot\! z \;\psi_1\left(\frac{r}{\e}-z\right)\!\cdot\! z\;\psi_2(r-\e z)\phi(t)\,dr dt.
\label{AS1ue_fix_3}
\end{eqnarray}
From Proposition \ref{existuniq_prop}, 
\begin{eqnarray}
\sup_{\e>0}\Vert u^{\e}\Vert_{L^{\infty}([0,T];\,L^p(\Omega)^3)}\leq \infty
\label{linftytimebd}
\end{eqnarray}
So from (\ref{AS1ue_fix_2}) and on application of H\"{o}lder's inequality, we see for some constants $C_1$ and $C_2$ that
\begin{eqnarray}
\label{AS1ue_fix_2_1}
 && \!\!\!\!\!\!\!\!\!\!\!\!\!\!\!|a^{n,\e}_1(z)|\leq
C_1\left(\int_{\mathbb{R}^3} \left\vert\chi_{\Omega}(r-\e z)-\zeta_n(r-\e z)\right\vert^{p'}\,dr\right)^{1/p'}\times\Vert u^\e\Vert_{L^\infty_{\ss{loc}}(\mathbb{R}_+;L^p(\Omega)^3)}\\
&& \!\!\!\!\!\!\!\!\!\!\!\!\!\!\!|a^{n,\e}_2(z)|\leq
C_2\Vert u^\e\Vert_{L^\infty_{\ss{loc}}(\mathbb{R}_+;L^p(\Omega)^3)}
\label{AS1ue_fix_2_1_b}
\end{eqnarray}
so  there is a constant $C$ such that $|a^\e(z)|<C$ for $\e>0$.
On the other hand, the second factor on the right hand side of (\ref{AS1ue_fix_2_1}) goes to zero uniformly in $\e$ as $n\rightarrow \infty$ and we conclude that for all $\e>0$ and
$z\in H_  {\delta    }(0)$,
\begin{eqnarray}
\label{AS1ue_fix_2_3}
\lim_{n\rightarrow \infty} a^{n,\e}_1(z)=0.
\end{eqnarray}
Now for $n$ fixed we see that as $\e\rightarrow 0$,
$
\zeta_n(r-\e z) \psi_2(r-\e z)\rightarrow \zeta_n(r) \psi_2(r)
$
uniformly. Therefore, we see from  (\ref{AS1ue_fix_3}) that
\begin{eqnarray}
\label{AS1ue_fix_3_1}
\nonumber
&&\lim_{\e\rightarrow 0} a^{n,\e}_2(z)\\
\nonumber
&&=\lim_{\e\rightarrow 0}
\int_{\mathbb{R}^3\times\mathbb{R}^{+}} \chi_\Omega(r)\zeta_n(r)\;\alpha\left(\frac{r}{\e}-z,\frac{r}{\e}\right)
u^{\e}(r,t)\!\cdot\! z \;\psi_1\left(\frac{r}{\e}-z\right)\!\cdot\! z\;\psi_2(r)\phi(t)\,dr dt\\
\nonumber
&& =\int_{ \Omega\times Y\times\mathbb{R}^{+}} \zeta_n(r)\alpha\left(y-z,y\right)
u(r,y,t)\!\cdot\! z \;\psi_1\left(y-z\right)\!\cdot\! z\;\psi_2(r)\phi(t)\,dr dy dt,\\
\end{eqnarray}
where in  the last step the fact that $(u^{\e})_{\e>0}$ two-scale converges to $u(r,y,t)$ was used. By taking
the limit as $n\rightarrow \infty$ in (\ref{AS1ue_fix_3_1}), we obtain
\begin{eqnarray}
\label{AS1ue_fix_3_2}
\nonumber
&&\!\!\!\!\!\!\!\!\!\!\!\!\!
\lim_{n\rightarrow \infty}\lim_{\e\rightarrow 0} a^{n,\e}_2(z)\\
&&\!\!\!\!\!\!\!\!\!\!\!\!\!
=\int_{\Omega\times Y\times\mathbb{R}^{+}}\!\!\!\!
 \alpha\left(y-z,y\right)
u(r,y,t)\!\cdot\! z \;\psi_1\left(y-z\right)\!\cdot\! z\;\psi_2(r) \phi(t)\, dr dy dt.
\end{eqnarray}
Equation (\ref{AS1ue_3}) now follows from  (\ref{AS1ue_fix_2_3}) and (\ref{AS1ue_fix_3_2})    since
\begin{eqnarray}
\label{AS1ue_fix_4}
\nonumber
\lim_{\e\rightarrow 0} a^{\e}(z) =
\lim_{n\rightarrow \infty}\lim_{\e\rightarrow 0}\, (a^{n,\e}_1(z)+a^{n,\e}_2(z)).
\end{eqnarray}

From (\ref{AS1ue_2}) and (\ref{AS1ue_3}), and by using Lebesgue's dominated convergence theorem applied to the sequence $(a^\e(z))_{\e>0}$, we obtain
\begin{eqnarray}
\label{AS1ue_4}
\nonumber
&&\!\!\!\!\!
\lim_{\e\rightarrow 0}\int_{\Omega\times \mathbb{R}^{+}}K^{\e}_{S,1}u^{\e}(x,t)\cdot \psi\left(x,\xe,t\right)\,dx dt\\
\nonumber
&&\!\!\!\!\!
 =\int_{H_  {\delta    }(0)}\frac{1}{|z|^3} \int_{\Omega\times Y\times\mathbb{R}^{+}}\!\!\!\!
\alpha\left(y-z,y\right)
u(r,y,t)\!\cdot\! z \;\psi_1\left(y-z\right)\!\cdot\! z\;\psi_2(r) \phi(t)\, dr dy dt dz\\
\nonumber
&&\!\!\!\!\!
 = \int_{\Omega\times\mathbb{R}^{+}}
 \int_{H_{\delta}(0)}\frac{1}{|z|^3}\int_{Y} \alpha\left(y-z,y\right)
u(r,y,t)\!\cdot\! z \;\psi_1\left(y-z\right)\!\cdot\! z\, dy dz \;\psi_2(r) \phi(t) dr dt, \\
\end{eqnarray}
where we have changed the order of integration in the last step.
After shifting the domain of integration in the inner integral of the right hand side of equation (\ref{AS1ue_4}), we obtain
\begin{eqnarray}
\label{AS1ue_periodicity_0}
\nonumber
&&\int_{Y}\alpha\left(y-z,y\right)
u(r,y,t)\!\cdot\! z \;\psi_1\left(y-z\right)\!\cdot\! z\, dy\\
\nonumber
&&=\int_{Y-z}\alpha\left(y,y+z\right)
u(r,y+z,t)\!\cdot\! z \;\psi_1\left(y\right)\!\cdot\! z\, dy\\
&&=\int_{Y}\alpha\left(y,y+z\right)
u(r,y+z,t)\!\cdot\! z \;\psi_1\left(y\right)\!\cdot\! z\, dy,
\end{eqnarray}
where in the last step the fact that the integrand is $Y$-periodic in $y$ was used.
Substituting (\ref{AS1ue_periodicity_0}) in equation (\ref{AS1ue_4}), then by changing the order of integration we obtain
\begin{eqnarray}
\label{AS1ue_5}
\nonumber
&&\!\!\!\!\!
\lim_{\e\rightarrow 0}\int_{\Omega\times \mathbb{R}^{+}}K^{\e}_{S,1}u^{\e}(x,t)\cdot \psi\left(x,\xe,t\right)\,dx dt\\
\nonumber
&&\!\!\!\!\!
 = \int_{\Omega\times\mathbb{R}^{+}}
 \int_{Y}\int_{H_{\delta}(0)}\alpha\left(y,y+z\right)
\frac{z\otimes z}{|z|^3}u(r,y+z,t) dz \cdot \psi_1(y) dy \;\psi_2(r) \phi(t) dr dt \\
\nonumber
&&\!\!\!\!\!
 = \int_{\Omega\times Y\times\mathbb{R}^{+}}
 B_{S,1}u(r,y,t)\cdot \psi(r,y,t) \, dr dy  dt, \\
\end{eqnarray}
where
\begin{eqnarray}
&&B_{S,1}u(x,y,t)=\int_{H_{\delta}(y)}\alpha\left(y,\yh\right)
\frac{(\yh-y)\otimes(\yh-y)}{\vert\yh-y\vert^{3}}\;u(x,\yh,t) d\yh
\label{B_S_1}
\end{eqnarray}
and $K^\e_{S,1}u^\e\stackrel{2}{\rightharpoonup}B_{S,1}u(x,y,t)$.

Next we evaluate the two-scale limit of $K^{\e}_{S,2} u^{\e}$. Let $\psi$ be a test function in $\mathcal{J}$. Then
by using (\ref{Ae_S2_1}), replacing $v(x)$ with $u^{\e}(x,t)$, we obtain
\begin{eqnarray}
\label{AS2ue_1}
\nonumber
&&\int_{\Omega\times \mathbb{R}^{+}}K^{\e}_{S,2}u^{\e}(x,t)\cdot \psi\left(x,\xe,t\right)\,dx dt\\
\nonumber
&& = \int_{\Omega\times \mathbb{R}^{+}}\int_{H_  {\delta    }(0)}\chi_\Omega(x+\e z) \alpha\left(\xe,\xe+ z\right)
\frac{z\otimes z}{\vert z\vert^{3}}\,dz\;  u^{\e}(x,t)\cdot \psi\left(x,\xe,t\right)\,dx dt.\\
\end{eqnarray}
The right hand side of (\ref{AS2ue_1}), after changing the order of integration, is equal to
\begin{eqnarray}
\label{AS2ue_2}
\int_{H_  {\delta    }(0)}b^\e(z) dz.
\end{eqnarray}
where $b^\e(z)$ is given by
\begin{eqnarray}
\label{AS2ue_2be}
b^\e(z)=\frac{1}{|z|^3}\int_{\Omega\times \mathbb{R}^{+}} \chi_\Omega(x+\e z) \alpha\left(\xe,\xe+ z\right)
u^{\e}(x,t)\!\cdot\! z \;\psi\left(x,\xe,t\right)\!\cdot\! z\, dx dt.
\end{eqnarray}
For future reference note that from Proposition \ref{existuniq_prop}, 
$\sup_{\e>0}\Vert u^{\e}\Vert_{L^{\infty}([0,T];\,L^p(\Omega)^3)} < \infty$ hence 
there is a constant $C$ such that the sequence $b^\e(z)$ is bounded above by
\begin{eqnarray}
|b^\e(z)|<C|z|^{-1}, \hbox{ for }\e>0.
\label{boundonb}
\end{eqnarray}
As before we approximate $\chi_\Omega$ by a sequence of smooth functions $\zeta_n$ such that $\zeta_n\rightarrow\chi_\Omega$ in $L^{p'}_{\ss{loc}}(\mathbb{R}^3)$
and write 
\begin{eqnarray}
b_n^\e(z)=\frac{1}{|z|^3}\int_{\Omega\times \mathbb{R}^{+}} \zeta_n(x+\e z) \alpha\left(\xe,\xe+ z\right)
u^{\e}(x,t)\!\cdot\! z \;\psi\left(x,\xe,t\right)\!\cdot\! z\, dx dt.
\label{smoothbe}
\end{eqnarray}

Next using the fact that $(u^{\e})_{\e>0}$ two-scale converges to $u(x,y,t)$, we see that for  $z\in H_{\delta}(0)$,
\begin{eqnarray}
\label{AS2ue_3}
\nonumber
&&\lim_{\e\rightarrow 0}b^\e(z)=\lim_{n\rightarrow\infty}\lim_{\e\rightarrow 0}b_n^\e(z)=
\frac{1}{|z|^3}\int_{\Omega\times Y\times \mathbb{R}^{+}} \alpha\left(y,y+ z\right)
u(x,y,t)\!\cdot\! z \;\psi\left(x,y,t\right)\!\cdot\! z\, dx dy dt.
\end{eqnarray}

From (\ref{AS2ue_1}), (\ref{AS2ue_2}) and (\ref{AS2ue_3}), and by using Lebesgue's dominated convergence theorem, we obtain
\begin{eqnarray}
\label{AS2ue_4}
\nonumber
&&\lim_{\e\rightarrow 0}\int_{\Omega\times \mathbb{R}^{+}}K^{\e}_{S,2}u^{\e}(x,t)\cdot \psi\left(x,\xe,t\right)\,dx dt\\
%\nonumber
&& = \int_{H_  {\delta    }(0)}\frac{1}{|z|^3}\int_{\Omega\times Y\times \mathbb{R}^{+}} \alpha\left(y,y+ z\right)
u(x,y,t)\!\cdot\! z \;\psi\left(x,y,t\right)\!\cdot\! z\, dx dy dt dz
\end{eqnarray}
By changing the order of integration and then  using the change of variables  $\yh=y+z$, we conclude that
\begin{eqnarray}
\label{AS2ue_5}
&&\lim_{\e\rightarrow 0}\int_{\Omega\times \mathbb{R}^{+}}K^{\e}_{S,2}u^{\e}(x,t)\cdot \psi\left(x,\xe,t\right)\,dx dt
=
 \int_{\Omega\times Y\times \mathbb{R}^{+}}B_{S,2}u(x,y,t)\cdot \psi\left(x,y,t\right)\,dx dy dt,
\end{eqnarray}
where
\begin{eqnarray}
&&B_{S,2}u(x,y,t)=
\int_{H_  {\delta    }(y)} \alpha\left(y,\yh\right)
\frac{(\yh-y)\otimes (\yh-y)}{\vert \yh-y\vert^{3}}\,d\yh\;  u(x,y,t),
\label{B_S_2}
\end{eqnarray}
and we conclude that $K^\e_{S,2}u^\e(x,t)\stackrel{2}{\rightharpoonup} B_{S,2}u(x,y,t)$.
Equation (\ref{AS_2scale}) follows on writing $B_S=B_{S,1}-B_{S,2}$. 

The operator $\rho^{-1}B_S$ is  a bounded operator on $L^p_{\ss{per}}(Y;C(\overline{\Omega})^3)$. This  follows from bounds on $\rho^{-1}B_{S,1}$ and $\rho^{-1}B_{S,2}$. Given any $w$ in $L^p_{\ss{per}}(Y;C(\overline{\Omega})^3)$ an application of Minkowski's inequality to\\ $\Vert\rho^{-1}B_{S,1}w(x,y)\Vert_{L^p_{\ss{per}}(Y;C(\overline{\Omega})^3)}$ shows that $\rho^{-1}B_{S,1}$ is bounded. The boundedness of $\rho^{-1}B_{S,2}$ easily follows from its definition.

\end{proof}

\subsection{Strong Approximation of Local Fields in Heterogeneous Peridynamic Media}
\label{ch_strongapproximation}

In this section it is shown that a rescaling in the $y$ variable of solution of the two-scale problem delivers a strong approximation to
the solution $u^\e(x,t)$ of the form $u(x,y,t)$. This is stated in the following theorem.
\begin{theorem}
\label{strongapproximation}
Let $u(x,y,t)$ be the solution of the two-scale problem given in Theorem \ref{two-scale-limit-eqn-thm} then
\begin{eqnarray}
\lim_{\e\rightarrow 0}\Vert u^\e(x,t)-u(x,\frac{x}{\e},t)\Vert_{L^p(\Omega)^3}=0, 
\end{eqnarray}
for every $t$ in $[0,T]$ and $\frac{3}{2}<p<\infty$.
\end{theorem}

From the perspective of computational mechanics the numerical effort necessary to discretize and solve
for $u(x,y,t)$ becomes much less expensive than direct numerical simulation for $u^\e(x,t)$ when the length scale of the microstructure $\e$ is sufficiently small relative to the computational domain. In view of Theorem \ref{strongapproximation} the numerical computation of $u(x,y,t)$ and the subsequent rescaling  $y=\frac{x}{\e}$
provides a viable multiscale numerical methodology. This topic is pursued in a forthcoming paper.

\begin{proof}
We start by writing the dynamics for the rescaled function $u(x,\frac{x}{\e},t)$. Making the substitution $y=\frac{x}{\e}$ in \eqref{two_scale_limit_full_form} delivers the following initial value problem for $u(x,\frac{x}{\e},t)$:  
\begin{equation}
\label{two_scale_limit_full_form_epsilon}
\renewcommand{\arraystretch}{0.7}
\begin{array}{lcl}
\partial_{t}^2 u(x,\frac{x}{\e},t) &=& \displaystyle\rho^{-1}(\frac{x}{\e})\int_{H_{\gamma}(0)}\chi_\Omega(x+\xi) \lambda
\frac{\xi\otimes\xi}{\vert\xi\vert^{3}}\left(\int_{Y}u(x+\xi,y',t)\,dy'-u(x,\frac{x}{\e},t)\right)\,d\xi\\
&&\\
& &+ \rho^{-1}(\frac{x}{\e})\displaystyle\int_{H_{\delta}(0)}\chi_\Omega(x+\e z) \alpha(\frac{x}{\e},\frac{x}{\e}+z)
\frac{z\otimes z}{\vert z\vert^{3}}\left(u(x,\frac{x}{\e}+z,t)-u(x,\frac{x}{\e},t)\right)\,dz \\
&&\\
& &+ \,\rho^{-1}(\frac{x}{\e})b(x,\frac{x}{\e},t),%\nonumber
\end{array}
\end{equation}
with $u(x,\frac{x}{\e},0)=u_0(x,\frac{x}{\e})$ and $\partial_t u(x,\frac{x}{\e},0)=v_0(x,\frac{x}{\e})$.

We subtract \eqref{two_scale_limit_full_form_epsilon} from \eqref{ue_operator_eq} to arrive at
the differential equation for the difference $e^\e(x,t)=u^\e(x,t)-u(x,\frac{x}{\e},t)$ given by
\begin{eqnarray}
\label{difference}
\partial^2_t e^{\e}(x,t)=A_S^\e e^\e(x,t)+A_L^\e e^\e(x,t)+d^\e(x,t)
\end{eqnarray}
with the homogeneous initial conditions $e^\e(x,0)=0$ and $\partial_t e^\e(x,t)=0$. Here the forcing term $d^\e(x,t)$ is of the form 
$d^\e(x,t)=\rho^{-1}(\frac{x}{\e})\left(d_{S,1}^\e+d_{S,2}^\e+d_L^\e\right)$ where
\begin{eqnarray}
d_{S,1}^\e=\int_{H_\delta(0)}\chi_\Omega(x+\e z)\alpha\left(\frac{x}{\e},\frac{x}{\e}+z\right)\frac{z\otimes z}{|z|^3}\left(u(x+\e z,\frac{x}{\e}+z,t)-u(x,\frac{x}{\e}+z,t)\right)\,dz,
\label{d_S_1}\\
d_{S,2}^\e=-\int_{H_\delta(0)}(1-\chi_\Omega(x+\e z))\alpha\left(\frac{x}{\e},\frac{x}{\e}+z\right)\frac{z\otimes z}{|z|^3}\left(u(x,\frac{x}{\e}+z,t)-u(x,\frac{x}{\e}+z,t)\right)\,dz,
\label{d_S_2}\\
d_L^\e=\int_{H_\gamma(0)}\chi_\Omega(x+\xi)\lambda\frac{\xi\otimes \xi}{|\xi|^3}\left(u(x+\xi,\frac{x+\xi}{\e},t)-\int_Yu(x+\xi,y',t)\,dy'\right)\,d\xi
\label{d_L}
\end{eqnarray}
The forcing term $d^\e(x,t)$ is regular and vanishes as $\epsilon\rightarrow 0$, this is stated in the following theorem.
\begin{theorem}
\label{errorrhs}
The forcing term $d^\e(x,t)$ belongs to $C([0,T];L^p(\Omega)^3)$ and
the sequence $(d^\e)_{\e}$ is uniformly bounded for $0\leq t\leq T$ where
\begin{eqnarray}
\sup_{\e>0}\sup_{t\in[0,T]}\Vert d^\e(x,t)\Vert_{L^p(\Omega)^3}<\infty,\hbox{ for }\frac{3}{2}< p<\infty,
\label{boundext}
\end{eqnarray}
\begin{eqnarray}
\lim_{\e\rightarrow 0}\Vert d^\e(x,t)\Vert_{L^p(\Omega)^3}=0,\hbox{ for all }t\in [0,T]\hbox{ and }\frac{3}{2}<p<\infty.
\label{limitofd}
\end{eqnarray}
\end{theorem}
We provide the proof of Theorem \ref{errorrhs} at the end of this section.
Since $A^\e$ is a bounded linear operator on $L^p(\Omega)^3$  it follows from Theorem \ref{errorrhs} and Proposition \ref{existuniq_prop} that the solution $e^\e(x,t)$ is explicitly given  by
\begin{eqnarray}
e^{\e}(x,t) = \int_0^t  \sum_{n=0}^{\infty} \frac{(t-\tau)^{2n+1}}{(2n+1)!}\; (A^{\e})^n d^{\e}(x,\tau)\,d\tau.
\label{errorrepresentation}
\end{eqnarray}
Thus
\begin{eqnarray}
\nonumber
\Vert e^{\e}(\cdot,t)\Vert_{L^p(\Omega)^3} &\leq&    \int_0^t \sum_{n=0}^{\infty} \frac{(t-\tau)^{2n+1}}{(2n+1)!}\;
\Vert(A^{\e})^n\Vert \;\Vert d^{\e}(\cdot,\tau)\Vert_{L^p(\Omega)^3}\,d\tau\\
\label{norm_e_ineq}
&\leq& \int_0^t \frac{1}{\sqrt{M}} \sinh{\left(\sqrt{M}(t-\tau)\right)}\;\Vert d^{\e}(\cdot,\tau)\Vert_{L^s(\Omega)^3}\,d\tau
\end{eqnarray}
where in the second inequality we have used the fact that $A^\e$ is bounded above by a positive constant $M>0$ independent of $\e$.
In view of Theorem \ref{errorrhs} we can apply the Lebesgue dominated convergence theorem to the right most inequality of \eqref{norm_e_ineq} to conclude that $\lim_{\e\rightarrow 0}\Vert e^{\e}(\cdot,t)\Vert_{L^p(\Omega)^3}=0$ and Theorem \ref{strongapproximation} is proved.
\end{proof}

We conclude this section by proving Theorem \ref{errorrhs}.
The theorem is proved by showing that each component of $d^\e$ given by $\rho^{-1}_\e d_{S,1}^\e$, $\rho^{-1}_\e d_{S,2}^\e$, $\rho^{-1}_\e d_{S,3}^\e$ belong to $C([0,T];L^p(\Omega)^3)$
and satisfy \eqref{boundext} and \eqref{limitofd}. We begin by showing that $\rho^{-1}_\e d_{S,1}^\e$ satisfies \eqref{boundext} and \eqref{limitofd} and that $\rho^{-1}_\e d_{S,1}^\e$ belongs to $C([0,T];L^p(\Omega)^3)$.
In what follows we use the basic estimate stated in the following lemma.
\begin{lemma}
\label{estimatelemma}
For any subset $S$ of $\Omega$ and $v(x,y,t)$ in $C([0,T];L_{\ss{per}}^p(Y;C(\overline{\Omega})^3))$ there exists a fixed integer independent of $\e$ denoted by $L>0$ for which
\begin{eqnarray}
&&\left(\int_S|v(x,\frac{x}{\e},t)|^p\,dx\right)^{1/p}\leq\left(\int_{S}\sup_{x'\in\Omega}|v(x',\frac{x}{\e},t)|^p\,dx \right)^{1/p}\nonumber\\
&&\leq L^{3/p}\Vert v\Vert_{L^p_{\ss{per}}(Y;C(\overline{\Omega})^3)}\leq L^{3/p}\Vert v\Vert_{C([0,T];L^p_{\ss{per}}(Y;C(\overline{\Omega})^3))}.
\label{Norms}
\end{eqnarray}
\end{lemma}
\begin{proof}
The proof is identical to the arguments used in the estimate \eqref{initialestu_0}.
\end{proof}

%In what follows we use the basic estimates stated in the following lemma.
%\begin{lemma}
%\label{estimatelemma}
%There exists a fixed integer independent of $\e$ denoted by $L>0$ for which
%\begin{eqnarray}
%\left(\int_{H_\gamma(0)}\sup_{x'\in\Omega}|u(x',\frac{x+\xi}{\e},t)|^p\,d\xi\right)^{1/p}&\leq&L^{3/p}\Vert %u\Vert_{L^p_{\ss{per}}(Y;C(\overline{\Omega})^3)}
%\label{Hgamma}
%\end{eqnarray}
%and for $z\in H_\delta(0)$,
%\begin{eqnarray}
%&&\left(\int_{\Omega}\sup_{x'\in\Omega}\left\{\chi_\Omega(x'+\e z)|u(x'+\e z,\frac{x}{\e}+z,t)-u(x',\frac{x}{\e}+z,t)|\right\}^p\,d\xi\right)^{1/p}\nonumber\\
%&&\leq L^{3/p}\Vert \chi_\Omega(x+\e z)(u(x+\e %z,y,t)-u(x,y,t))\Vert_{L^p_{\ss{per}}(Y;C(\overline{\Omega})^3)}\nonumber\\
%\label{Omega}
%\end{eqnarray}
%\end{lemma}

We begin by showing that $\rho^{-1}_\e d_{S,1}^\e$ satisfies \eqref{boundext} and \eqref{limitofd} and that $\rho^{-1}_\e d_{S,1}^\e$ belongs to $C([0,T];L^p(\Omega)^3)$.
Let $\overline{\alpha}=\max_{y,y'\in Y}\rho^{-1}(y)\alpha(y,y')$ and estimate 
\begin{eqnarray}
\Vert\rho^{-1}_\e d_{S,1}^\e\Vert_{L^p(\Omega)}&\leq&\left(\int_\Omega\left(\int_{H_\delta(0)}\chi_\Omega(x+\e z)\frac{\overline{\alpha}}{|z|}|u(x+\e z,\frac{x}{\e},t)-u(x,\frac{x}{\e},t)|\,dz\right)^p\,dx\right)^{1/p}
\nonumber\\
&\leq&\int_{H_\delta(0)}\frac{\overline{\alpha}}{|z|}\left(\int_\Omega\chi_\Omega(x+\e z)|u(x+\e z,\frac{x}{\e},t)-u(x,\frac{x}{\e},t)|^p\,dx\right)^{1/p}\,dz\nonumber\\
&\leq&\int_{H_\delta(0)}\frac{\overline{\alpha}}{|z|}\left(\int_\Omega\sup_{x'\in\Omega}\left\{\chi_\Omega(x'+\e z)|u(x'+\e z,\frac{x}{\e},t)-u(x',\frac{x}{\e},t)|\right\}^p\,dx\right)^{1/p}\,dz\nonumber\\
&\leq&L^{3/p}\int_{H_\delta(0)}\frac{\overline{\alpha}}{|z|}f_\e(z,t)\,dz,
\label{estfor_d_s_1}
\end{eqnarray}
where $f_\e(z,t)$ is given by
\begin{eqnarray}
f_\e(z,t)=\Vert\chi_\Omega(x'+\e z)(u(x'+\e z,\frac{x}{\e},t)-u(x',\frac{x}{\e},t))\Vert_{L^p_{\ss{per}}(Y;C(\overline{\Omega})^3)}.
\label{feps}
\end{eqnarray}
Here the second inequality in \eqref{estfor_d_s_1} follows from the Minkowski inequality and the last inequality
in \eqref{estfor_d_s_1} follows from Lemma \ref{estimatelemma}.
Next we show that $\lim_{\e\rightarrow 0}|f_\e(z,t)|=0$. To see this write
\begin{eqnarray}
g_\e(y,z,t)=\sup_{x\in\Omega}\left\{\chi_\Omega(x+ \e z)|u(x+\e z,y,t)-u(x,y,t)|\right\}
\label{g}
\end{eqnarray}
and note that 
\begin{itemize}
\item
$g_\e \rightarrow 0$ for almost every $y\in Y$, $t\in [0,T]$, and $z\in H_\delta(0)$,
\item
$0\leq g_\e(y,z,t)\leq 2\sup_{x\in\Omega}|u(x,y,t)|$,
\end{itemize}
and $\lim_{\e\rightarrow 0}|f_\e(z,t)|=0$ follows from the Lebesgue dominated convergence theorem since $u$ belongs to $L^p_{\ss{per}}(Y;C(\overline{\Omega})^3)$ for every $t\in[0,T]$.
Observe next that 
\begin{eqnarray}
\sup_{\e>0}|f_\e(z,t)|\leq 2\Vert u\Vert_{L^p_{\ss{per}}(Y;C(\overline{\Omega})^3)}\leq 2\Vert u\Vert_{C([0,T];L^p_{\ss{per}}(Y;C(\overline{\Omega})^3))}.
\label{estforf_e}
\end{eqnarray}
Hence we apply the Lebegue dominated convergence theorem again to find that 
\begin{eqnarray}
\lim_{\e\rightarrow 0}\Vert\rho^{-1}_\e d_{S,1}^\e\Vert_{L^p(\Omega)}=0
\label{vanising_d_S_1}
\end{eqnarray}
and application of \eqref{estforf_e} to the last line of \eqref{estfor_d_s_1} gives
\begin{eqnarray}
\sup_{t\in[0,T]}\sup_{\e>0}\Vert\rho^{-1}_\e d_{S,1}^\e\Vert_{L^p(\Omega)}<\infty.
\label{boundtimed_S_1}
\end{eqnarray}

Given $0\leq t<t'\leq T$ we apply Minkowski's inequality together with Lemma \ref{estimatelemma} to obtain the estimate 
\begin{eqnarray}
\Vert\rho^{-1}_\e d_{S,1}^\e(t)-\rho^{-1}_\e d_{S,1}^\e(t')\Vert_{L^p(\Omega)}&\leq&2\overline{\alpha}\left(\int_{H_\delta(0)}|z|^{-1}\,dz\right)
\Vert u(t)-u(t')\Vert_{L^p_{\ss{per}}(Y;C(\overline{\Omega})^3)}.
\label{continuity_d_S_1}
\end{eqnarray}
Since $u$ belongs to $C^2([0,T];L_{\ss{per}}^p(Y;C(\overline{\Omega})^3))$ the  estimate \eqref{continuity_d_S_1} implies that $d_{S,1}^\e(t)$ belongs to \\
$C([0,T];L_{\ss{per}}^p(Y;C(\overline{\Omega})^3))$.

Now we discuss the boundedness, continuity  and convergence of $\rho^{-1}d_{S,2}^\e$. 
The overall approach to demonstrating these properties for $\rho^{-1}d_{S,2}^\e$ is the same as before. 
Here we point out that the mechanism that drives $\rho^{-1}d_{S,2}^\e$ to zero with $\e$ is the point wise convergence  $1-\chi_\Omega(x+\e z)\rightarrow 0$ for every $x\in \Omega$. The norm bounds and continuity properties of $u(x,y,t)$ 
are then used as before to establish the continuity properties, boundedness and convergence of the sequence $(\rho^{-1}d_{S,2}^\e)_{\e}$.

The overall approach to demonstrating properties for the sequence $(\rho^{-1}d_{L}^\e)_{\e}$ is also the same, however
there are some distinctions that arise in the proof of convergence. In what follows we outline the proof of convergence pointing out that the continuity proof and bounds are established as before.
We begin noting that $u$ belongs to $\mathcal{Q}_p$ with $\frac{3}{2}<p<\infty$ hence from Proposition \ref{smooth_two_scale}
\begin{eqnarray}
u(x,\frac{x}{\e},t)\stackrel{2}{\rightharpoonup}u(x,y,t),
\label{testu}
\end{eqnarray}
and from Proposition \ref{two_scale_time_weak_convg} it follows that for any test function $\psi(x)\in L^{p'}(\Omega)$
with $\frac{1}{p'}+\frac{1}{p}=1$ that
\begin{eqnarray}
\int_{\Omega}\psi(x) u(x,\frac{x}{\e},t)\,dx\rightarrow\int_\Omega\psi(x)\int_Y u(x,y,t)\,dy\,dx, \hbox{ as }\e\rightarrow 0.
\label{weaku}
\end{eqnarray}
We write
\begin{eqnarray}
\Vert\rho^{-1}_\e d_{L}^\e\Vert_{L^p(\Omega)}&=&\left(\int_\Omega|h_\e(x)|^p\,dx\right)^{1/p},
\label{formula_norm_d_L}
\end{eqnarray}
where
\begin{eqnarray}
h_\e(x)&=&\int_{H_\gamma(0)}\chi_\Omega(x+\xi)\lambda\frac{\xi\otimes\xi}{|\xi|^3}\left(u(x+\xi,\frac{x+\xi}{\e},t)-\int_Y u(x+\xi,y',t)\,dy'\right)\,d\xi.
\label{h_e}
\end{eqnarray}
We apply \eqref{weaku} noting that  $\psi(\xi)=\chi_{\Omega}(x+\xi)\frac{\xi\otimes\xi}{|\xi|^3}$ belongs to $L^{p'}$ for $p'<3$ to find that 
\begin{eqnarray}
\lim_{\e\rightarrow 0}h_\e(x)=0.
\label{h_goesto_0}
\end{eqnarray}
%provided that $u$ belongs to $\mathcal{L}$ with $p>\frac{3}{2}$. 

Application of H\"older's inequality to the right hand side of \eqref{h_e} for $p'<3$ gives the upper bound
\begin{eqnarray}
|h_\e(x)|
&\leq&\lambda\left(\int_{H_\gamma(0)}|\xi|^{-p'}\,d\xi\right)^{1/p'}\left(\int_{H_\gamma(0)}\chi_\Omega(x+\xi)|u(x+\xi,\frac{x+\xi}{\e},t)|^p\,d\xi\right)^{1/p}\nonumber\\
&&+\left(\int_{H_\gamma(0)}|\xi|^{-p'}\,d\xi\right)^{1/p'}\left(\int_{H_\gamma(0)}\chi_\Omega(x+\xi)\left|\int_Y u(x+\xi,y',t)dy'\right|^p\,d\xi\right)^{1/p}.
\label{upperbdh_e}
\end{eqnarray}
Applying Lemma \ref{Norms} to the first term on the right hand side of \eqref{upperbdh_e}, Minkowski's inequality to the second term followed with H\"olders inequality  delivers
the inequality
\begin{eqnarray}
|h_\e(x)|\leq C\Vert h\Vert_{L^p_{\ss{per}}(Y;C(\overline{\Omega})^3)},
\label{stability_h}
\end{eqnarray}
where $C$ is a positive constant independent of $\e$.
From \eqref{h_goesto_0} and \eqref{stability_h} it now follows from the Lebesgue bounded convergence theorem that
\begin{eqnarray}
\lim_{\e\rightarrow 0}\Vert\rho^{-1}_\e d_{L}^\e\Vert_{L^p(\Omega)}=0.
\label{d_L_e_goesto_0}
\end{eqnarray}
The continuity and boundedness properties for $\rho^{-1}_\e d_{L}^\e$ follow along lines similar to the previous arguments.

\section{Homogenized Peridynamics}
\label{sec_macro_micro}

The strong approximation $u(x,\frac{x}{\e},t)$ admits a natural decomposition into a continuous macroscopic component and a possibly discontinuous fluctuating component. The macroscopic component
$u^H(x,t)$ is obtained by projecting out the spatial fluctuations 
and the corrector $r(x,\frac{x}{\e},t)$ containing the possibly discontinuous fluctuations is given by the remainder, i.e., 
\begin{eqnarray}
u(x,\frac{x}{\e},t)=u^H(x,t)+r(x,\frac{x}{\e},t),
\label{splitting}
\end{eqnarray}
where
\begin{eqnarray}
u^H(x,t)=\langle u\rangle\equiv\int_Y\,u(x,y,t)\,dy
\label{avg}
\end{eqnarray}
and
\begin{eqnarray}
r(x,\frac{x}{\e},t)=u(x,\frac{x}{\e},t)-u^H(x,t).
\label{fluct}
\end{eqnarray}

The weak convergence expressed by Proposition \ref{two_scale_time_weak_convg}
gives
\begin{eqnarray}
\lim_{\e\rightarrow 0}\frac{1}{|V|}\int_V\, u^\e(x,t)\,dx=\lim_{\e\rightarrow 0}\frac{1}{|V|}\int_V\,u(x,\frac{x}{\e},t)\,dx=\frac{1}{|V|}\int_V\,u^H(x,t)\,dx,\label{macroavg}
\end{eqnarray}
and
\begin{eqnarray}
\lim_{\e\rightarrow 0}\frac{1}{|V|}\int_V\, r(x,\frac{x}{\e},t)\,dx=0.
\label{microprojectout}
\end{eqnarray}

It is evident from \eqref{macroavg} that the macroscopic component $u^H$ tracks the average or upscaled behavior
of the actual field $u^\e$. Conversely the macroscopic or ``averaged'' observations of the actual deformation $u^\e$ will track the dynamics of $u^H$. Thus it is of compelling interest  to obtain an explicit evolution equation for $u^H$ in order to qualitatively account for observations made at macroscopic length scales.
In what follows we show that averaging the two-scale peridynamic equations over the $y$ variable delivers a coupled system for the macroscopic and microscopic components $u^H(x,t)$ and $r(x,y,t)$. This coupling is seen to impart a history dependence on the evolution of $u^H$. We express this memory effect explicitly by eliminating $r$
and recovering an integro-differential equation in both space and time for $u^H$. 

In what follows we set $u^H(t)=u^H(\cdot,t)$ and $r(t)=r(\cdot,t)$
and we denote spatial averages of fields $v(x,y,t)$ taken over the $y$ variable by $\langle v\rangle(t)\equiv\int_Y v(x,y,t)\,dy$. Let the constant $3\times 3$ matrix $K$ be defined by
\begin{eqnarray}
K=\lambda\int_{h_\gamma(0)}\frac{\xi\otimes\xi}{|\xi|^3}\,d\xi
\label{K}
\end{eqnarray}
and the coupled dynamics for the evolution of $u^H(t)$ and $r(t)$ is given by the following theorem. 
\begin{theorem}
\label{couple_dynam}
\begin{eqnarray}
\ddot{u}^H(t)&=&\langle\rho^{-1}\rangle K_L u^H(t)+\langle\rho^{-1}B_S r\rangle(t)-K\langle\rho^{-1}r\rangle(t)+\langle\rho^{-1} b\rangle(t),\label{M}\\
\ddot{r}(t)&=&\left(\rho^{-1}-\langle\rho^{-1}\rangle\right)K_L u^H(t)+\left(\rho^{-1}B_S r(t)-\langle\rho^{-1}B_S r\rangle(t)\right)\nonumber\\
&-&K\left(\rho^{-1} r(t)-\langle\rho^{1} r\rangle(t)\right)+\left(\rho^{-1}b(t)-\langle\rho^{-1} b\rangle(t)\right),
\label{c_d}
\end{eqnarray}
with initial conditions $u^H(0)=\langle u_0\rangle$, $\dot{u}^H(0)=\langle v_0\rangle$, $r(0)=u_0-\langle u_0\rangle$, and  $\dot{r}(0)=v_0-\langle v_0\rangle$.

\end{theorem}
\begin{proof}
We write $u(x,y,t)=u^H(x,t)+r(x,y,t)$ and substitute this into the two-scale peridynamic equation \eqref{two_scale_limit_full_form}. Next multiply both sides of \eqref{two_scale_limit_full_form} by $\rho^{-1}$
and then take the average both sides of \eqref{two_scale_limit_full_form} with respect to the $y$ variable.
The equation for $u^H$ given by \eqref{M} follows noting that $\langle r\rangle(t)=0$ and 
\begin{eqnarray}
\langle\ddot{r}\rangle(t)=\partial^2_t\langle r\rangle=0,
\label{derivavg}
\end{eqnarray}
where the operations of differentiation  and integration commute since since $u\in C^2([0,T];L^p_{\ss{per}}(Y;C(\overline{\Omega})^3))$.
The equation \eqref{c_d} follows on substitution of \eqref{M} in \eqref{two_scale_limit_full_form}.
\end{proof}

Now we obtain an evolution equation for $u^H$ by eliminating $r$ from the system given by \eqref{M} and \eqref{c_d}.
Let
\begin{eqnarray}
\mathcal{C}r(t)&=&\rho^{-1}B_S r(t)-\langle\rho^{-1}B_S r\rangle(t)-K\left(\rho^{-1} r(t)-\langle\rho^{-1} r\rangle(t)\right),
\label{C}
\end{eqnarray}
and \eqref{c_d} becomes
\begin{eqnarray}
\ddot{r}(t)&=&\mathcal{C}r(t)+\left(\rho^{-1}-\langle \rho^{-1}\rangle\right)K_L u^H(t)+\rho^{-1}b(t)-\langle\rho^{-1} b\rangle(t).
\label{CMm}
\end{eqnarray}
Since \eqref{CMm} is linear we set $r=v+w$ where
\begin{eqnarray}
\ddot{v}(t)=\mathcal{C}v(t)+\left(\rho^{-1}-\langle \rho^{-1}\rangle\right)K_L u^H(t),
\label{veq}
\end{eqnarray}
with initial conditions $v(0)=0$, $\dot{v}(0)=0$ and 
\begin{eqnarray}
\ddot{w}(t)=\mathcal{C}w(t)+\rho^{-1} b(t)-\langle \rho^{-1}b\rangle(t),
\label{weq}
\end{eqnarray}
with initial conditions $w(0)=\hat{u}_0=u_0-\langle u_0\rangle$ and $\dot{w}(0)=\hat{v}_0=v_0-\langle v_0\rangle$.

Proceeding as before one finds that $\mathcal{C}$ is a linear operator on $L^p_{\ss{per}}(Y;C(\overline{\Omega})^3)$
and $v(t)$ and $w(t)$ are given by
\begin{eqnarray}
v(t)&=&\left(\sqrt{\mathcal{C}}\right)^{-1}\int_0^t\sinh{\left((t-\tau)\sqrt{\mathcal{C}}\right)}\left(\rho^{-1}-\langle\rho^{-1}\rangle\right)K_L u^H(\tau)\,d\tau\label{iv}\\
w(t)&=&\cosh{t\sqrt{\mathcal{C}}}\hat{u}_0+\left(\sqrt{\mathcal{C}}\right)^{-1}\sinh{t\sqrt{\mathcal{C}}}\hat{v}_0\nonumber\\
&+&\left(\sqrt{\mathcal{C}}\right)^{-1}\int_0^t\sinh{\left((t-\tau)\sqrt{\mathcal{C}}\right)}\left(\rho^{-1}b(\tau)-\langle\rho^{-1} b\rangle(\tau)\right)\,d\tau.
\label{iw}
\end{eqnarray}
Let
\begin{eqnarray}
\mathcal{K}=\langle\rho^{-1}B_S r\rangle(t) - K\langle \rho^{-1 }b\rangle(t),
\label{Kk}
\end{eqnarray}
then substitution of $r=v+w$ in \eqref{M} gives the homogenized integro-differential equation for $u^H(t)$ given by
the following theorem.
\begin{theorem}
The homogenized deformation $u^H(t)$ is the solution of the integro-differential equation in space and time given by 
\begin{eqnarray}
\langle\rho^{-1}\rangle^{-1}\ddot{u}^H(t)&=&K_L u^H(t)+
\langle\rho^{-1}\rangle^{-1}\mathcal{K}\left(\sqrt{\mathcal{C}}\right)^{-1}\int_0^t\sinh{\left((t-\tau)\sqrt{\mathcal{C}}\right)}\left(\rho^{-1}-\langle\rho^{-1}\rangle\right)K_L u^H(\tau)\,d\tau\nonumber\\
&+&\langle\rho^{-1}\rangle^{-1}\left(\mathcal{K} w(t)+\langle\rho^{-1} b\rangle(t)\right),
\label{id_u_H}
\end{eqnarray}
with the initial conditions $u^H(0)=\langle u_0\rangle$ and $\dot{u}^H(0)=\langle v_0\rangle$.
The force generated by the homogenized deformation $f^H(t)=f^H(\cdot,t)$ is given by the history dependent constitutive law
\begin{eqnarray}
f^H(t)=K_L u^H(t)+
\langle\rho^{-1}\rangle^{-1}\mathcal{K}\left(\sqrt{\mathcal{C}}\right)^{-1}\int_0^t\sinh{\left((t-\tau)\sqrt{\mathcal{C}}\right)}\left(\rho^{-1}-\langle\rho^{-1}\rangle\right)K_L u^H(\tau)\,d\tau.
\label{hc}
\end{eqnarray}
\end{theorem}
This equation shows that the evolution law for the homogenized deformation $u^H$ is history dependent.

\end{document}